\documentclass[11pt]{article}
\usepackage{amsmath}
\usepackage{amssymb}
\usepackage{amsthm}

\DeclareMathOperator*{\argmin}{arg\,min}
\usepackage{graphicx}
\usepackage{enumerate}
\usepackage{makecell}
\usepackage{hyperref}
\usepackage{xcolor}
\usepackage{natbib}
\usepackage{url} % not crucial - just used below for the URL 
\usepackage{fullpage}

\newtheorem{theorem}{Theorem}
\newtheorem{lemma}{Lemma}
\newtheorem{definition}{Definition}
\newtheorem{corollary}{Corollary}
\newtheorem{remark}{Remark}

\begin{document}

\begin{center}

{\bf{\LARGE{Robust W-GAN-Based Estimation Under Wasserstein Contamination}}}

\vspace*{.25in}

%{\large{
 \begin{tabular}{cc}
Zheng Liu & Po-Ling Loh \\
\texttt{zliu577@wisc.edu} & \texttt{pll28@cam.ac.uk} \\
Department of Statistics & Department of Pure Maths and Mathematical Statistics \\
University of Wisconsin-Madison & University of Cambridge
\end{tabular}
%}}

\vspace*{.2in}

January 2021

\vspace*{.2in}

\end{center}

%\usepackage{authblk}
%\title{Robust W-GAN-Based Estimation Under Wasserstein Contamination}
%\author[1]{Zheng Liu}
%\author[2]{Po-Ling Loh}
%\affil[1]{University of Wisconsin-Madison}
%\affil[2]{University of Cambridge}

%\maketitle
%\bibliographystyle{natbib}

\newcommand{\E}{\ensuremath{\mathbb{E}}}
\newcommand{\thetastar}{\ensuremath{\theta^*}}
\newcommand{\betastar}{\ensuremath{\beta^*}}
\newcommand{\Sigmastar}{\ensuremath{\Sigma^*}}
\newcommand{\thetahat}{\ensuremath{\widehat{\theta}}}
\newcommand{\real}{\ensuremath{\mathbb{R}}}
\newcommand{\tr}{\ensuremath{\operatorname{tr}}}
\newcommand{\mprob}{\ensuremath{\mathbb{P}}}
\newcommand{\etastar}{\ensuremath{\eta^*}}

%%%%%%%%%%%%%%%%%%%%%%%%%%%%%%%%%%%%%%%%%%%%%%%%%%%%%%%%%%%%%%%%%%%%%%%%%%%%%%

\begin{abstract}
Robust estimation is an important problem in statistics which aims at providing a reasonable estimator when the data-generating distribution lies within an appropriately defined ball around an uncontaminated distribution. Although minimax rates of estimation have been established in recent years, many existing robust estimators with provably optimal convergence rates are also computationally intractable. In this paper, we study several estimation problems under a Wasserstein contamination model and present computationally tractable estimators motivated by generative adversarial networks (GANs). Specifically, we analyze properties of Wasserstein GAN-based estimators for location estimation, covariance matrix estimation, and linear regression and show that our proposed estimators are minimax optimal in many scenarios. Finally, we present numerical results which demonstrate the effectiveness of our estimators.
\end{abstract}

%\noindent%
%{\it Keywords:} robust estimation, Wasserstein contamination, minimax rate, location estimation, covariance matrix estimation, regression, generative adversarial networks
%\vfill

%\newpage

%%%%%%%%%%%%%%%%%%%    Section 1.  Introduction     %%%%%%%%%%%%%%%%
\section{Introduction}
\label{sec:intro}

Robust estimation aims at providing a reasonable estimator for a functional of a data-generating distribution under small perturbations of the distribution. In the classical setting of Huber's $\epsilon$-contamination model \citep{huber:1964}, the goal is to estimate a parameter $\theta$ from i.i.d.\ data drawn from a mixture distribution $(1-\epsilon) P_{\theta} + \epsilon Q$, where $P_\theta$ is the uncontaminated distribution, $Q$ is an arbitrary distribution, and $\epsilon$ is the contamination proportion. Much rich theory has been established and estimators with optimal minimax rates have been found \citep{collins_wiens_1985, chen_gao_ren_2016, chen_gao_ren_2018}. However, Huber's contamination model only allows a fraction of data samples to be corrupted (on average). In our work, we consider a different contamination model which can perturb all data samples. %Specifically, we consider perturbations which are within distance $\epsilon$ in terms of the Wasserstein distance.

The Wasserstein distance \citep{villani_cedric_2009}, defined as an optimal transport cost between two distributions, is becoming increasingly popular in machine learning and statistics \citep{gao2016distributionally, tolstikhin2017wasserstein, titouan2019sliced, WonEtal19}. Under the Wasserstein contamination model, we aim to estimate $\theta$ from i.i.d.\ data $X_i \sim P$, where $P$ is a perturbed distribution such that the Wasserstein distance between $P$ and the true distribution $P_\theta$ is bounded by $\epsilon$. In this paper, we study robustness of estimators under the Wasserstein contamination model in the settings of location estimation, covariance matrix estimation, and regression.

Recent work by \cite{Gao1} derived fascinating connections between generative adversarial networks (GANs) \citep{NIPS2014_5423} and robust estimation, suggesting new computational tools for robust estimation. GANs were first proposed in deep learning to learn the distribution of a data set, which could in turn be used to generate synthetic samples that are indistinguishable from true samples by the human eye.
%A GAN consists of a generator, which is used to generate fake data, and a discriminator, which distinguishes true and fake data.
GANs essentially attempt to minimize the probability divergence between a true distribution and learned distribution \citep{NIPS2016_6066}. For a classical GAN, the divergence is the Jensen-Shannon divergence; for the total variation GAN (TV-GAN) \citep{NIPS2016_6066}, the divergence is the total variation distance; and for the Wasserstein GAN (W-GAN) \citep{pmlr-v70-arjovsky17a}, the divergence is the Wasserstein distance.
At the outset, it may be somewhat surprising that the two models---one from the field of statistics and one from the field of deep learning---could be related.
%Furthermore, since robust estimators are often computationally intractable in moderate to high dimensions, GAN-based estimators can alleviate this problem due to modern advances in neural network training.
\cite{Gao1} showed that under Huber's contamination model, a location estimator based on the TV-GAN has the same minimax rate as the Tukey median, which is known to be minimax optimal for robust estimation.

Under our Wasserstein contamination model, it is natural to study the performance of W-GANs for robust estimation.
%We propose a general W-GAN based estimator which can be applied to location estimation, covariance matrix estimation, and regression.
W-GANs have an empirical advantage over classical GANs in the sense that training is more stable and they can avoid the problem of ``mode collapse," which has led to an uptick of interest in machine learning in recent years \citep{adler2018banach, cao2019multi, liu2019wasserstein}.
We present a general technique for upper-bounding the minimax rate of W-GAN-based estimators, and derive a general lower bound on the minimax rate via the modulus of continuity. As we will see, the upper and lower bounds match in many estimation settings of interest, indicating that the W-GAN-based estimator achieves the optimal minimax rate.
%We also consider sparse location estimation and sparse covariance estimation, and conclude that the W-GAN-based estimator is nearly minimax optimal in these settings, as well.
Our results are generally derived under the assumption that the uncontaminated distribution is Gaussian, and we also provide extensions to the case of elliptical distributions.

The rest of the paper is organized as follows: Section~\ref{sec:background} provides a detailed background on the Wasserstein contamination model and GAN model. Section~\ref{sec:upperbound} provides a general upper bound on the minimax rate, and Section~\ref{sec:lowerbound} provides a general lower bound on the minimax rate. Sections~\ref{sec:loc},~\ref{sec:cov}, and~\ref{sec:reg} apply the general upper and lower bounds to location estimation, covariance matrix estimation, and linear regression problems. Section~\ref{sec:num} provides numerical simulations. We conclude the paper in Section~\ref{sec:dis} with a discussion of future work. 

\textbf{Notation:} We use $\|v\|_2$ to denote the $\ell_2$-norm of a vector, $\|v\|_0$ for the $\ell_0$-norm of a vector, $\|A\|_2$ for the spectral norm of a matrix, and $\|A\|_F$ for the Frobenius norm of a matrix. We use $\lambda_{\min}(A)$ and $\lambda_{\max}(A)$ to denote the smallest and largest eigenvalues. We use $\|f(x)\|_L$ to denote the Lipschitz constant of a function $f$. We use $[p]$ to denote the set $\{1,2, \dots, p\}$. For two positive sequences $\{a_n\}$ and $\{b_n\}$, we use $a_n \lesssim b_n$ or $a_n = O(b_n)$ to denote the fact that $a_n \leq C b_n$ for some constant $C>0$, and also write $b_n = \Omega(a_n)$. We use $a_n \asymp b_n$ to denote the fact that both  $a_n \lesssim b_n$ and $b_n \lesssim a_n$ hold simultaneously.

%%%%%%%%%%%%%%%%%%%    Section 2. Background     %%%%%%%%%%%%%%%%

\section{Background}
\label{sec:background}

In this section, we provide details about the contamination model we will consider in this paper. We also provide more details about GAN-based estimators and related work.

%%%%%

\subsection{Robust estimation model}

We first recall the framework of robust estimation introduced by \cite{huber:1964}. Assume that $P_{\thetastar}$ is the true distribution and $Q$ is an arbitrary distribution. In Huber's contamination model, we have $n$ i.i.d.\ observations $X_i \sim (1-\epsilon)P_{\thetastar} + \epsilon Q$, where $\epsilon \in (0,1)$ is fixed and possibly unknown.
Although the minimax risk of various estimation problems under Huber's contamination model has recently been derived~\citep{chen_gao_ren_2016, chen_gao_ren_2018}, the question of obtaining estimators which are \emph{computationally tractable} in high (or even moderate) dimensions has remained largely open.

Motivated by recent work \cite{Gao1}, which used tools from deep learning to devise computationally tractable robust estimators, we pivot our attention to the problem of robust estimation under a Wasserstein contamination model. Recall that the Wasserstein distance of order $q$ between two distributions $P$ and $Q$ is defined as  $W_q(P, Q) = \left(\inf_{\Pi(P,Q)}\E_{X,Y\sim\Pi}\|X-Y\|_2^q\right)^{\frac{1}{q}}$, where $\Pi(P,Q)$ is any coupling between $P$ and $Q$ \citep{villani_cedric_2009}. In our paper, we focus on the case $q = 1$, and simply write $W(P, Q)$. Suppose $P_{\thetastar}$ is the true distribution from a parametric family. In the context of robust estimation, our goal will be to obtain an estimator $\theta$ which achieves the minimax risk
\begin{equation}
\inf_{\theta} \sup_{\thetastar \in \Theta, P: W(P_{\thetastar}, P) \le \epsilon} \E_{X_i \sim P} L(\thetastar, \theta), \label{minimax1}
\end{equation}
where $\epsilon$ is the level of contamination, $L$ is the loss, and the expectation is taken over i.i.d.\ samples from the perturbed distribution $P$.

Robust estimation under Wasserstein contamination has only been lightly studied. One recent paper \citep{Zhu1} proposed an estimator based on minimum distance functionals \citep{Donoho2}, where the empirical distribution $P_n$ is first projected into a distribution class $\mathcal{G}$ according to a distance measure $D$ between distributions. The final estimator is defined in terms of the projected distribution $\widehat{Q} = \argmin_{Q\in \mathcal{G}}D(Q, P_n)$. For example, in the case of location estimation, the estimator $\widehat{\theta}$ is simply the mean of $\widehat{Q}$. \cite{Zhu1} showed how to design the projection set $\mathcal{G}$ and distance function $D$ so that the resulting estimator would have good finite-sample properties. Consequently, the restrictions on $\mathcal{G}$ depend on the specific estimation problem. Furthermore, the assumptions they impose on the class of uncontaminated distributions is more general than the Gaussian/elliptical classes we consider in this paper, leading to a somewhat more complicated analysis.\footnote{As an example, for second moment estimation under a Wasserstein-1 contamination model, \cite{Zhu1} only assumed that the true distribution $p^*$ has bounded $k^{\text{th}}$ moments.
%$\mathcal{G}=\mathcal{G}_{sec}(4\sigma^{1+1/(k-1)}(7\tilde{\epsilon})^{1-1/(k-1)}, 7\tilde{\epsilon})$, where
%\begin{equation*}
%\mathcal{G}_{sec}(\rho, \eta) = \left\{p: \sup_{\|v\|_2=1,r\in \mathbb{F}} |\E_p|v^TX|^2 - \E_r|v^TX|^2|\leq \rho\right\}.
%\end{equation*}
%Here, $\mathbb{F}$ is a distribution class of ``friendly perturbations," defined in their paper. 
%The definition of the corresponding projection set $\mathcal{G}$ is not short, so we do not rewrite it here.
The distance function $D$ is a weakened Wasserstein-1 distance, defined by $D(p,q) := \sup_{u\in \mathcal{U}} |\E_pu(X) - \E_qu(X)|$, where
\begin{equation}
\mathcal{U} = \{\max(0,v^Tx - a)|v\in \mathbb{R}^p, \|v\|_2\leq 1, a\in \mathbb{R}\} \cup \{v^Tx|v\in \mathbb{R}^d, \|v\|_2 \leq 1\}. \label{EqZhu2}
\end{equation}
%and define $D(p,q) := \sup_{u\in \mathcal{U}} |\E_pu(X) - \E_qu(X)|$.
%\label{Zhu3}
\cite{Zhu1} then provided an upper bound on the maximum risk of the above estimator. For second moment estimation, assuming the true distribution has a bounded $k^{\text{th}}$ moment, they obtained a bound of $O\left(\left(\sqrt{\frac{p}{n}}+\epsilon\right)^{1-\frac{1}{(k-1)}}\right)$. }
Due to the complexity of the set $\mathcal{G}$, however, it is not practically feasible to implement the projection estimator---indeed, \cite{Zhu1} focused on deriving theoretical upper bounds for their estimator, rather than obtaining computationally feasible estimators. Furthermore, \cite{Zhu1} did not derive lower bounds for minimax risk under Wasserstein contamination, leaving the question of optimality unaddressed.

In our paper, we propose a GAN-based estimator and provide both upper and lower bounds on the minimax risk, thus showing optimality of our estimators. We also provide an explicit algorithm for obtaining the estimators (based on training a GAN). Interestingly, as we will see later, our proposed estimator shares some similarities with the  estimator proposed by \cite{Zhu1}, although one is derived starting from GAN models and the other is derived from the perspective of minimum distance functionals.

%%%%%

\subsection{Robust objective and WGAN algorithm}
\label{objective}
Under the Wasserstein contamination model, first suppose we wish to optimize a population-level version of the ``robust risk," given by
%Then the estimator $\theta$ in equation~\eqref{minimax1} is a function of $P$, and the risk $\mathbb{E}L(\thetastar, \theta)$ is replaced by the deterministic quantity $L(\thetastar, \theta)$.
%and $\widehat{\theta}$ only depends on $P$ and is not random. 
%To optimize the simplified objective
%\begin{equation}
%\label{EqnNonRandom}
%\inf_\theta \sup_{\thetastar \in \Theta, P: W(P_{\thetastar}, P) \le \epsilon} L(\thetastar, \theta),
%\end{equation}
%one idea is to first minimize the distance between the true distribution and estimated distribution, i.e. $D(P_{\thetastar}, P_{\theta})$, under some distance metric $D$. 
\begin{equation}
\label{EqnDObj}
\inf_\theta \sup_{\thetastar \in \Theta, P: W(P_{\thetastar}, P) \le \epsilon} D(P_{\thetastar}, P_\theta),
\end{equation}
where $D$ is some distance function and $\theta$ is a function of $P$.
One can hope that minimizing an empirical version of the objective~\eqref{EqnDObj} will lead to an estimator which achieves the minimax risk~\eqref{minimax1} measured in terms of $L$---indeed, that is what we find in our results.
%Note the similarity between the objectives~\eqref{EqnDObj} and~\eqref{minimax1} above; we do not have to take an expectation here, since we are not working with a finite sample.

For the Wasserstein contamination model, if we choose $D$ to be the Wasserstein-1 distance, then by the triangle inequality, since $W(P_{\thetastar}, P) \leq \epsilon$, we have
\begin{equation}
D(P_{\thetastar}, P_{\theta}) \leq D(P, P_{\theta}) + \epsilon, \label{triang}
\end{equation}
for any estimator $\theta$. Thus, we have
\begin{equation}
\sup_{\thetastar \in \Theta, P: W(P_{\thetastar}, P) \le \epsilon} D(P_{\thetastar}, P_\theta) \leq \sup_{\thetastar \in \Theta, P: W(P_{\thetastar}, P) \le \epsilon} D(P, P_\theta) + \epsilon.
%= \sup_P D(P, P_\theta) + \epsilon
\end{equation}
%The last equality holds since our estimator $\theta$ is a function only of $P$. 
This leads to the following relaxation of the objective~\eqref{EqnDObj}:
\begin{equation*}
\inf_\theta \sup_{\thetastar \in \Theta, P: W(P_{\thetastar}, P) \le \epsilon} D(P, P_\theta).
\end{equation*}
%minimize $\sup_{P_\theta, P}D(P_\theta, P_\eta)$, we only need to minimize $D(P, P_\eta)$, for every $P$.
Recalling that $\theta$ is a function of $P$ alone, this motivates the idea of obtaining an estimator $\thetahat$ by directly minimizing $D(P, P_\theta)$ for any $P$.
%Note that the above argument based on the triangle inequality also holds if we consider perturbation under total variation distance and let $D$ be the total variation distance.

In practice, we only have access to the empirical distribution $P_n$ rather than the population distribution $P$, so we might instead minimize $D(P_n, P_\theta)$. However, the convergence rate of $D(P, P_n)$ may be slow: For example, when $D$ is the Wasserstein distance, we have $\mathbb{E}W(P, P_n) \gtrsim n^{-1/p}$ for any continuous distribution $P$ on $\mathbb{R}^p$~\citep{dudley1969speed}.
%Furthermore, if $D$ is the total variation distance and $P$ is a continuous distribution, then $D(P_n, P_\theta)=1$ for any $\theta$, so minimizing $D(P_n, P_\theta)$ is meaningless.
One way to remedy this problem is to replace the distance $D$ with a weaker distance $\tilde{D}$. We show how to do this when $D$ is the Wasserstein distance. By Kantorovich duality, we have the following duality form of the (rescaled) Wasserstein-1 distance between two distributions $P_1$ and $P_2$ \citep{villani_cedric_2009}:
\begin{equation}
B \cdot W(P_1, P_2)=\sup_{\|f(x)\|_L\leq B} \left\{\E_{P_1} f(X) - \E_{P_2} f(X)\right\},\label{Wgan}
\end{equation}
where $B > 0$ is a scale parameter. The supremum is taken over all Lipschitz-$B$ continuous functions. By this duality result, we  see that we can obtain a relaxation of the Wasserstein distance by taking a supremum over a subset $\mathcal{D}$ of Lipschitz-$B$ continuous functions.
%Denote
%\begin{equation*}
%\tilde{W}(P_1, P_2)=\sup_{\mathcal{D}} \left\{\E_{P_1} f(X) - \E_{P_2} f(X)\right\}.
%\end{equation*}
This leads us to our final estimation algorithm
\begin{equation}
\widehat{\theta} = \argmin_\theta \sup_{f \in \mathcal{D}} \left\{\E_{P_n} f(X) - \E_{P_\theta} f(X)\right\}. \label{model}
\end{equation}

\begin{remark}
\label{RemZhu}
\cite{Zhu2} also considered the problem of minimizing $D(P_n, P_\theta)$, for a general distribution distance function. They relax the distance function $D$ to another distance $D'$ which they called an \textit{admissible distance}, with the desired property that $D'$ can distinguish $P_n$ and $P_\theta$ as well as $D$, but has better convergence properties than $D(P, P_n)$.
%For more discussion about the work in this paper vs.\ the work of \cite{Zhu2}, see Section~\ref{SecConnections} below.
\end{remark}

%%%%%

%%%%%%

%\subsection{Connection between robust estimation and GANs}
%\label{SecConnections}

%Comparing the GAN-based estimator~\eqref{tv3} and \eqref{w2} with \eqref{model}, we immediately find that 
In fact, the robust estimator~\eqref{model} is nothing but a GAN-based estimator (see Appendix~\ref{AppGAN} for more background). In other words, we can use GAN-based estimators for robust estimation.
%Now we give a brief literature review on connection between robust estimation and GAN. 
%Robust estimation and GAN models are very popular in their respective fields of statistics and deep learning. However, connections between these topics have only recently been discovered.

\subsection{Related work}

The first work describing such a connection between robust estimation and GAN models was \cite{Gao1}. In that paper, the authors considered Gaussian location estimation under Huber's contamination model, and proved that TV-GAN-based estimators and JS-GAN-based estimators can achieve the minimax rate in this scenario. In follow-up work, \cite{Gao2} generalized the GAN-based estimator to covariance matrix estimation, also under Huber's contamination model. They introduced a ``learning via classification" framework based on a notion of a proper scoring rule.
%However, essentially, they use GAN-based estimators.
Using different scoring rules, they arrived at formulations based on TV-GANs or JS-GANs, and proved that the GAN-based estimators thus obtained can achieve the minimax rate for covariance matrix estimation.

\cite{WuEtal20} also considered Huber's contamination model and generalized the TV-GAN and JS-GAN-based estimators from \cite{Gao1} to a general $f$-GAN-based estimator. Similar to our setting, they also studied a W-GAN-based estimator; however, in contrast to our work, they considered Huber's contamination model and proved that the estimator is not optimal for robust location estimation. They also did not perform relaxation of the objective as we have discussed in Section~\ref{objective}, underscoring the need for both steps in order to derive a minimax optimal estimator. We also note that \cite{WuEtal20} considered a sparse location estimation problem in their paper, resembling our setup in Section~\ref{spsloc}, but again, their focus was on $f$-GANs and Huber contamination.

We also mention a work by \cite{Zhu2} that provided additional insights regarding connections between robust estimation and GANs. The authors also viewed GAN models as minimizing a certain distance between distributions (cf.\ Remark~\ref{RemZhu} above).
%They proposed minimizing $D(P, P_\eta)$ for some suitable distance $D$ for robust estimation. And they also pointed out the need of relaxation of distance $D$ when we replace $P$ with empirical distribution $P_n$.
However, their goal was somewhat different, as they sought to devise a projection algorithm which achieves an error (defined in a population-level sense) close to the oracle error given by the best approximation of the population-level distribution in the generator class. Furthermore, they focused on Wasserstein-2 contamination and did not provide explicit bounds for the estimation error of the parameters in their work. In this paper, we consider the Wasserstein-1 contamination model and provide minimax rates for both location and covariance matrix estimation, as well as sparse location and covariance estimation and regression.

Our work leverages GANs as a computational tool, and operates under the assumption that we can find an optimal solution of the GAN model using state-of-the-art deep learning techniques. Although W-GANs have been shown empirically to have good stability during training~\citep{pmlr-v70-arjovsky17a}, we note that theoretical guarantees for neural network training still remain elusive, and convergence to a global optimum is technically not guaranteed.

%To be more specific, in robust estimation, we would like to minimize the maximum risk,  $\sup_{P_\theta,P} \E L(\theta, \widehat{\theta})$. For simplicity, suppose we can obtain an estimate from the population distribution $P$ instead of the empirical distribution $P_n$, so we do not need to include the expectation in $\E L(\theta, \widehat{\theta})$, and 
%$L(\theta, \widehat{\theta})$ is actually equivalent to some divergence between $P_{\widehat{\theta}}$ and $P_\theta$, i.e., $D(P_\theta, P_{\widehat{\theta}})$. Let us denote $P_{\widehat{\theta}}$ by $Q$, where $Q$ depends on $P$. Then by \eqref{triang}, the maximum risk is bounded by
%\begin{equation*}
%\sup_{\theta, P} D(P_\theta, Q) \leq \sup_{P} D(P,Q) + \epsilon.
%\end{equation*}
%It is not hard to see that the right-hand side is minimized by $Q^*(P) = \argmin_Q D(P, Q)$. Thus, we need to minimize $D(P, Q)$, which is exactly what GAN models seek to accomplish. 
%
%From another point of view, if we can only observe perturbed distributions, the best we can do is to find a distribution from the true distribution class which has the smallest distance to the perturbed distribution. Thus leads naturally to the same formulation as in the GAN model.

%%%%%%%%%%%%%%%%%%%%%   Section 3.  A general upper bound     %%%%%%

\section{A general upper bound}
\label{sec:upperbound}

In this section, we will prove a general upper bound on the risk for our W-GAN-based estimators. This will serve as a foundation for all proofs under different robust estimation problems, e.g., location estimation, covariance matrix estimation, and linear regression. We first specify the exact estimator we will use in practice.

In the W-GAN model, the function class in \eqref{w2} contains all Lipschitz-$B$ continuous functions. In practice, multilayer neural networks are used to realize classification functions. Accordingly, we define a function class consisting of neural network functions which are also Lipschitz-$B$ continuous, and use it for our estimation procedure~\eqref{model}. To make a neural network model Lipschitz continuous, we need to add some constraints to our network structure and weight parameters.

Specifically, we consider neural network models with the first layer involving the sigmoid activation function $\sigma(u)=\frac{1}{1+e^{-u}}$, and all hidden layers involving the ReLU activation function $\text{ReLU}(u) = \max(u,0)$. For the output layer, we choose not to use an activation function, since under this condition, our function class is symmetric: for any function $f \in \mathcal{D}$, we have $-f \in \mathcal{D}$. We add norm constraints on the weight parameters---the reason will be clear in Lemma~\ref{LemNNLip}.  A node in the first layer takes the form
\begin{equation}
\mathcal{D}^{(1)} = \left\{f(x)=\sigma(w^T x + b): w\in \mathbb{R}^p, \|w\|_2 \leq B,  b\in \mathbb{R}\right\}, \label{func1}
\end{equation}
and layer $l+1$, for $1\leq l < L-1$, takes the form
\begin{equation}
\mathcal{D}^{(l+1)} = \left\{f^{(l+1)}(x)=\text{ReLU}\left(\sum\limits_{i=1}^{d^{(l)}} w_i f^{(l)}_i(x)\right): \sum_{i=1}^{d^{(l)}} |w_i| \leq 1, f^{(l)}_i(x) \in \mathcal{D}^{(l)}\right\},
\label{func2}
\end{equation}
where $d^{(l)}$ is the number of nodes in the $l^{\text{th}}$ layer. The last layer takes the form
\begin{equation}
\mathcal{D} = \left\{f^{(L)}(x) = \sum\limits_{i=1}^{d^{(L-1)}} w_i f^{(L-1)}_i(x): \sum_{i=1}^{d^{(L-1)}} |w_i| \leq 1, f^{(L-1)}_i(x) \in \mathcal{D}^{(L-1)}\right\}. \label{func3}
\end{equation}
We use $\mathcal{D}(B, L)$ to denote the function class defined by equations~\eqref{func1},\eqref{func2}, and \eqref{func3}.

%Now we can see the connection of our W-GAN based estimator and the estimator in  \cite{Zhu1}. By the duality of the Wasserstein distance, we have
%$$\sup_{f \in \mathcal{D}} \E_{P_n}f(X) - \E_{P_\eta}f(X) \leq W(P_n,P_\eta).$$
%Thus, our algorithm \eqref{model} can be viewed as minimizing a relaxed version of Wasserstein distance $\tilde{W}(P_n, P_\eta)$, which coincides with idea in the algorithm of \cite{Zhu1}.

\begin{remark}
Comparing our algorithm with the algorithm of \cite{Zhu1} (cf.\ equation~\eqref{EqZhu2}), we can see that the latter function class can be viewed as a one-layer neural network, whereas our function class $\mathcal{D}(B,L)$ corresponds to a general multilayer neural network. Furthermore, the work of \cite{Zhu1} addressed parameter estimation in a setting with fewer distributional assumptions, requiring them to take a minimum over a more complex distribution class $\mathcal{G}$. Our work takes a minimum over the whole true distribution class, which is much simpler and does not involve a projection step. Further note that our upper bounds are proved in completely different ways.
\end{remark}

%Note that one can design other neural network models which also satisfy the Lipschitz condition, and all our theorems will also hold for those function classes. 
We also introduce a sparse neural network model which will be used for sparse estimation problems. Define the nodes in the first layer as 
\begin{equation}
\tilde{\mathcal{D}}^{(1)} = \left\{f(x)=\sigma(w^T x + b): w\in \mathbb{R}^p, \|w\|_2 \leq B, \|w\|_0 \leq k,  b\in \mathbb{R}\right\}, \label{func4}
\end{equation}
where $k < p$. The hidden layers and output layer are the same as in equations~\eqref{func2} and \eqref{func3}. We use $\mathcal{D}_{sp}(B, L, k)$ to denote the resulting neural network function class. Lemma~\ref{LemNNLip} and Corollary~\ref{CorNNLip} in Appendix~\ref{AppNN} show that the function classes satisfy the Lipschitz condition.

Our main theorem provides an upper bound between the true parameter and the estimated parameter in the ``Wasserstein distance" sense:

\begin{theorem}
\label{Upper1}
Suppose $P_{\thetastar}$ is the true distribution, $P$ satisfies $W(P,P_{\thetastar}) \leq \epsilon$, and $X_i \stackrel{i.i.d.}{\sim} P$. Let $\widehat{\theta}$ denote the estimator from the W-GAN model~\eqref{model}, with function class $\mathcal{D} = \mathcal{D}(B, L)$. Then for any $\delta > 0$, with probability at least $1-\delta$, we have
$$\sup_{f \in \mathcal{D}(B, L)} |\E_{P_{\thetastar}}f(X) - \E_{P_{\widehat{\theta}}}f(X)| \leq 2 C\left(2^{L}\sqrt{\frac{p}{n}} + \sqrt{\frac{2\log(1/\delta)}{n}}\right)+ \frac{B}{2} \epsilon.$$
\end{theorem}

The proof of Theorem~\ref{Upper1}, contained in Appendix~\ref{AppUpper1}, borrows ideas from the proof of Theorem 3.1 of \cite{Gao1}. We can obtain an analogous bound, proved in Appendix~\ref{AppUpper2}, for the sparse neural network function class:

\begin{theorem}
\label{Upper2}
Under the same assumptions in Theorem~\ref{Upper1}, with the function class changed to $\mathcal{D} = \mathcal{D}_{sp}(B, L, k)$, for any $\delta > 0$, with at least $1-\delta$ probability, we have
$$\sup_{f \in \mathcal{D}_{sp}(B,L,k)} |\E_{P_{\thetastar}}f(X) - \E_{P_{\widehat{\theta}}}f(X)| \leq 2 C\left(2^{L}\sqrt{\frac{k\log\frac{p}{k}}{n}} + \sqrt{\frac{2\log(1/\delta)}{n}}\right)+ \frac{B}{2}\epsilon.$$
%uniformly over all $\theta$ and $P$.
\end{theorem}

%The proof of Theorem~\ref{Upper2} is contained in Appendix~\ref{AppUpper2}. 

%We can generalize Theorem~\ref{Upper1} in another way, with a broader function class. The first layer is:
%\begin{equation}
%\widehat{\mathcal{D}}^{(1)} = \{f(x)=\sigma(w^T x + b): w\in \mathbb{R}^p, \|w\|_2 \leq B,  b\in \mathbb{R}\}. \label{func5}
%\end{equation}
%where we constrain the 2-norm by a constant $B > 0$ rather than 1. The hidden layers and output layer are the same as \eqref{func2} and \eqref{func3}. It is not hard to see that this new function class is also bounded (by 2) and Lipschitz continuous with constant $\frac{B}{4}$. Also note by changing function class from Lipschitz-1 continuous to Lipschitz-$\frac{B}{4}$ continuous, we are essentially minimizing a scaled Wasserstein distance, scaled by $\frac{B}{4}$. Similar to Theorem~\ref{Upper1}, we have the following theorem:
%
%\begin{theorem}
%\label{Upper3}
%Under same assumptions in theorem~\ref{Upper1}, with function class changed to \eqref{func5}, \eqref{func2} and \eqref{func3}, for any $\delta > 0$, with at least $1-\delta$ probability, we have
%$$\sup_{f \in \mathcal{D}} |\E_{P_\theta}f(X) - \E_{P_{\widehat{\theta}}}f(X)| \leq 2 C\left(2^{L}\sqrt{\frac{p}{n}} + \sqrt{\frac{2\log(1/\delta)}{n}}\right)+ \frac{B}{2}\epsilon, $$
%uniformly over all $\theta$ and $P$.
%\end{theorem}
%
%The proof of Theorem~\ref{Upper3} is contained in Appendix~\ref{AppUpper3}.

In the sections below, we will show how to translate the upper bound on $\sup_{f \in \mathcal{D}} |\E_{P_{\thetastar}}f(X) - \E_{P_{\widehat{\theta}}}f(X)|$ into a bound on $\|\thetahat - \thetastar\|$ in the settings of location estimation, covariance estimation, and linear regression.

%This is a general bound which holds for any setup, location estimation, covariance estimation and regression. From this bound, we can get the upper bound on distance between true and estimated parameter in the following sections.

%%%%%%%%%%%%%%%%%%%%%   Section 4.  A general lower bound     %%%%%%
\section{A general lower bound}
\label{sec:lowerbound}

In this section, we present a general lower bound on the minimax risk for the Wasserstein contamination model. The bound is based on the modulus of continuity, which was first proposed and studied by \cite{Donoho1, Donoho2} and \cite{Donoho3}. In later sections, we derive lower bounds for several problems, including location estimation, covariance matrix estimation, and regression, from the main theorem. These lower bounds will help establish minimax optimality of our proposed W-GAN-based algorithms.

\begin{definition}
Let $L$ be a loss function. For $\epsilon > 0$ and a parameter space $\Theta$, we define the \emph{modulus of continuity}
$$ m(\epsilon, \Theta) \triangleq  \sup_{W(P_{\theta_1}, P_{\theta_2})\leq \epsilon, \theta_1, \theta_2 \in \Theta} L({\theta_1}, {\theta_2}).$$
\end{definition}

The modulus of continuity can be thought of as the maximum loss the perturbation can induce within an $\epsilon$ ball. For the Wasserstein contamination model, we have the following probabilistic lower bound on the minimax risk~\eqref{minimax1} based on the modulus of continuity:

\begin{theorem}
\label{ThmModulus}
Suppose $L$ is a loss function which satisfies the triangle inequality. Also suppose there is some $\mathcal{M}(0)$ such that for $\epsilon=0$, the inequality
$$ \inf_{\widehat{\theta}} \sup_{\thetastar \in \Theta, P:W(P_{\thetastar},P)\leq \epsilon} \mathbb{P}(L(\thetastar, \widehat{\theta}) \geq \mathcal{M}(\epsilon)) \geq c $$
holds for some constant $c>0$. Then for any $\epsilon \in [0,1]$, the same holds for $\mathcal{M}(\epsilon) \asymp \mathcal{M}(0) \vee m(\epsilon, \Theta)$.
\end{theorem}

The proof of Theorem~\ref{ThmModulus} is contained in Appendix~\ref{AppThmModulus}.

\begin{remark}
Theorem~\ref{ThmModulus} resembles Theorem 5.1 in \cite{chen_gao_ren_2016}. However, we consider Wasserstein contamination, whereas they considered Huber's contamination model. In their paper, they also defined the modulus of continuity slightly differently, using the total variation distance, i.e., $ m(\epsilon, \Theta) =  \sup\limits_{TV(P_{\theta_1}, P_{\theta_2})\leq \frac{\epsilon}{1-\epsilon}} L({\theta_1}, {\theta_2})$.
%Note that the consistency between the metrics used in our definitions of the contamination model and modulus of continuity (i.e., Wasserstein distance) make our proof of Theorem~\ref{ThmModulus} shorter and simpler than the proof of the corresponding result in \cite{chen_gao_ren_2016}. 
\end{remark}

Theorem~\ref{ThmModulus} says that for any estimator $\widehat{\theta}$, there always exists a true distribution $P_\theta$ and contaminated distribution $P$ such that the loss is at least $\mathcal{M}(\epsilon)$ with some nonzero probability $c$. Thus, the worst-case risk for the estimator is at least $\Omega(\mathcal{M}(\epsilon))$.
We will see that in many cases, the W-GAN-based estimator achieves this risk, showing that it is minimax optimal.

%\begin{remark}
%Note that Theorem~\ref{ThmModulus} could also be generalized to other contamination models. From the proof, we can see that if we consider contamination models under other distances $D$, e.g., the total variation distance or general Wasserstein-$p$ distance, and we define the modulus of continuity accordingly, then as long as we can find a distribution $P_\eta$ which satisfies $D(P_{\eta_1}, P_\eta) \leq \epsilon$ and $D(P_{\eta_2}, P_\eta) \leq \epsilon$, the same proof would imply similar lower bounds for the other contamination model. However, calculating the modulus of continuity is very problem-specific, and depends on the distance used in the definition (which depends on the contamination model being studied) and the loss function (which depends on the specific estimation problem: location, covariance matrix, regression, etc.). So even if the lower bound has a similar form for different contamination models (e.g., Huber or Wasserstein), the exact form of the bound may be nontrivial to compute.
%\end{remark}

\begin{remark}
\cite{Zhu1} also used the modulus of continuity to characterize the minimax risk under Wasserstein contamination. However, their notion of risk was a population-level version, and they showed that their projection algorithm (defined by projecting the true distribution into their class $\mathcal{G}$) achieved the minimax risk. They then computed bounds on the modulus continuity for various estimation problems and choices of $\mathcal{G}$.
%However, they use the modulus of continuity as an upper bound on the proposed algorithm (the projection algorithm), and then they design a projection set and distance function such that the modulus of continuity is bounded, so that the maximum risk of their algorithm is bounded.
\end{remark}

%%%%%%%%%%%%%%%%%%%%%   Section 5.  Location estimation   %%%%%%

\section{Location estimation}
\label{sec:loc}

In this section, we consider robust location estimation problems~\citep{lehmann2006theory}. We provide upper and lower bounds for the W-GAN-based estimator and show that the estimator achieves optimal convergence rate. We consider location estimation for Gaussian and elliptical distributions.

%%%%%

\subsection{Gaussian location estimation}

In Gaussian location estimation, we assume that the true data-generating distribution is multivariate Gaussian. We observe i.i.d.\ samples from the distribution and would like to use the samples to estimate the true mean of the Gaussian distribution. We have the following upper bound for the W-GAN-based estimator:

\begin{theorem}
\label{LocUpper}
Assume the true distribution is $P_{\thetastar} = \mathcal{N}(\thetastar, I_p)$, with parameter space $\Theta = \real^p$. Suppose the contaminated distribution $P$ satisfies $W(P,P_{\thetastar}) \leq \epsilon$, and $X_i \stackrel{i.i.d.}{\sim} P$. Consider the function class $\mathcal{D}=\mathcal{D}(B, L)$. Assume $2^L \cdot \frac{p}{n} + \epsilon^2 \leq c$ for some sufficiently small constant $c$. Then the estimator $\widehat{\theta}$ from the W-GAN model~\eqref{model} satisfies
$$ \|\widehat{\theta} - \thetastar\|_2 \lesssim  2^L \sqrt{\frac{p}{n}} \vee \epsilon, $$
with probability at least $1 - e^{-(p+n\epsilon^2)}$, where the constant prefactor depends only on $B$.
\end{theorem}

The proof of Theorem~\ref{LocUpper} is contained in Appendix~\ref{AppLocUpper}.

%\begin{remark}
%As we can see from the proof, the constant before $\sqrt{\frac{p}{n}} \vee \epsilon$ increases with $L$ and decreases with $B$.
%The fact that the bound in Theorem~\ref{LocUpper} increases with $L$ may seem counter-intuitive, since more layers usually lead to better estimation in deep learning. However, with more layers, the function class $\mathcal{D}(B, L)$ approximates the full Lipschitz function class more closely and the objective becomes more similar to the Wasserstein distance. However, as discussed in Section~\ref{objective}, we actually want to relax the Wasserstein distance in order to obtain an optimal convergence rate.
%\end{remark}

\begin{remark}
\label{rmkLoc1}
We can also see from the proof of Theorem~\ref{LocUpper} that the assumption of the true distribution being Gaussian can be replaced by a more general assumption that the distribution of $u^T(X-\thetastar)$ is symmetric around 0, for all unit vectors $u$, when $X\sim P_{\thetastar}$.
\end{remark}

The proof of Theorem~\ref{LocUpper} yields an exponential dependence on $L$ in the prefactor of the upper bound. Thus, choosing $L=1$ leads to the tightest rates. This may seem surprising at first, since approximation results for neural networks generally require the number of layers to be large~\citep{Cyb89}. On the other hand, recall that our estimators are derived using a \emph{relaxation} of the class of $B$-Lipschitz functions (cf.\ Section~\ref{objective}), and it is not inherently clear that a tighter relaxation should lead to a more robust estimator. The fact that the upper bound increases with $L$ comes from the fact that Theorem~\ref{LocUpper} uses the uniform concentration bound in Theorem~\ref{Upper1}, which naturally becomes looser as the size of the function class increases.

Next, we derive a lower bound for robust location estimation under Wasserstein contamination, using Theorem~\ref{ThmModulus}:

\begin{theorem}
\label{LocLower}
Assume the true distribution is $P_{\thetastar} = \mathcal{N}(\thetastar, I_p)$, with parameter space $\Theta=\real^p$. Suppose the contaminated distribution $P$ satisfies $W(P,P_{\thetastar}) \leq \epsilon$, and $X_i \stackrel{i.i.d.}{\sim} P$. Then there exist constants $C, c>0$ such that
$$ \inf_{\widehat{\theta}} \sup_{\thetastar \in \Theta, P:W(P_{\thetastar},P)\leq \epsilon} \mathbb{P}\left(\|\thetastar - \widehat{\theta}\|_2 \geq C\left(\sqrt{\frac{p}{n}} \vee \epsilon\right)\right) \geq c.$$
\end{theorem}

The proof of Theorem~\ref{LocLower} is contained in Appendix~\ref{AppLocLower}.

\begin{remark}
We can see that the upper bound for our W-GAN-based estimator matches the lower bound, which implies that the W-GAN based estimator is minimax optimal. Under Huber's contamination model, the minimax rate is also $\Theta(\frac{p}{n} \vee \epsilon^2)$ and can be attained by Tukey's median or the TV-GAN-based estimator~\citep{Gao1}.
\end{remark}

%%%%%

\subsection{Sparse Gaussian location estimation}
\label{spsloc}

In sparse Gaussian location estimation \citep{sparse_loc2, sparse_loc3, sparse_loc1}, the true data-generating distribution is Gaussian and the true mean vector is sparse. Under the Wasserstein contamination model, we observe i.i.d.\ samples from the contaminated distribution and use them to estimate the sparse mean. We define the parameter space
\begin{equation}
\Theta = \left\{\theta \in \mathbb{R}^p: \|\theta\|_0 \leq k < \frac{p}{2}\right\}.\label{space:spsGauss}
\end{equation}

We have the following upper and lower bounds, proved in Appendices~\ref{AppSparseLocUpper} and~\ref{AppSparseLocLower}:

\begin{theorem}
\label{SparseLocUpper}
Assume the true distribution is $P_{\thetastar} = \mathcal{N}(\thetastar, I_p)$, where $\thetastar$ lies in the parameter space~\eqref{space:spsGauss}. Suppose the contaminated distribution $P$ satisfies $W(P,P_{\thetastar}) \leq \epsilon$, and $X_i \stackrel{i.i.d.}{\sim} P$. Consider the function class $\mathcal{D}=\mathcal{D}_{sp}(B, L, k)$. Assume $2^L\cdot \frac{k\log\frac{p}{k}}{n} + \epsilon^2 \leq c$ for some sufficiently small constant $c$. Then the estimator $\widehat{\theta}$ from the W-GAN model~\eqref{model} satisfies
$$ \|\widehat{\theta} - \thetastar\|_2 \lesssim  2^L \sqrt{\frac{k\log\frac{p}{k}}{n}} \vee \epsilon, $$
with probability at least $1 - e^{-(p+n\epsilon^2)}$, and the constant prefactor only depends on $B$.
\end{theorem}

%The proof of Theorem~\ref{SparseLocUpper} is contained in Appendix~\ref{AppSparseLocUpper}.

\begin{remark}
Note that when $k \geq \frac{p}{2}$, we need to use the normal first layer \eqref{func1} in our neural network, leading to the same upper bound as in Theorem~\ref{LocUpper}.
\end{remark}

\begin{theorem}
\label{SparseLocLower}
Assume the true distribution is $P_{\thetastar} = \mathcal{N}(\thetastar, I_p)$, where $\thetastar$ lies in the parameter space~\eqref{space:spsGauss}. Suppose the contaminated distribution $P$ satisfies $W(P,P_{\thetastar}) \leq \epsilon$, and $X_i \stackrel{i.i.d.}{\sim} P$. Then there exist constants $C, c>0$ such that
$$ \inf_{\widehat{\theta}} \sup_{\thetastar \in \Theta, P:W(P_{\thetastar},P)\leq \epsilon} \mathbb{P}\left(\|\thetastar - \widehat{\theta}\|_2 \geq C\left(\sqrt{\frac{k\log\frac{p}{k}}{n}} \vee \epsilon\right)\right) \geq c. $$
\end{theorem}

%The proof of Theorem~\ref{SparseLocLower} is contained in Appendix~\ref{AppSparseLocLower}.

Comparing the conclusions of Theorems~\ref{SparseLocUpper} and~\ref{SparseLocLower}, we can see that the bounds again match each other, showing the optimality of the W-GAN-based estimator. 
%\end{remark}

%%%%%

\subsection{Elliptical distributions}
\label{SecElliptical}

In location estimation for elliptical distributions \citep{srivastava1989stein}, the true distribution is an elliptical distribution. We use the same setup as in \cite{Gao1}. The stochastic representation is given by $X = \thetastar + \xi A^* U$,
where $\thetastar \in \mathbb{R}^p$ is the location parameter, $\xi$ is a random variable controlling the shape, $U \in \mathbb{R}^p$ is uniformly distributed on the unit sphere, and $A^* \in \mathbb{R}^{p \times p}$ is a deterministic matrix.

Note that for any unit vector $v\in \mathbb{R}^p$, the distribution of variable $S = v^T\xi U$ is symmetric and does not depend on $v$. We denote the density function by $h^*(s)$. Since $\xi A^* = (c\xi)(\frac{1}{c}A^*)$ for any $c\neq 0$, the parameters $(\xi, A^*)$ can only be identified up to a constant. To make the parameters identifiable, we normalize $\xi$ so that the density of $S$ satisfies
\begin{equation}
\label{EqnRestrict}
\int \sigma'(s)h^*(s) ds = 1.
\end{equation}
%Then the parameterization $(\thetastar, h^*, A^*)$ of the elliptical distribution is identifiable. 
Finally, we denote the corresponding elliptical distribution by $\mathcal{E}(\thetastar, h^*, A^*)$.

\begin{theorem}
\label{EllipLoc}
Assume the true distribution is $P_{\etastar} =  \mathcal{E}(\thetastar, h^*, A^*)$, where $\thetastar$ lies in the parameter space $\Theta=\real^p$, and $\lambda_{\min}(A^*A^{*T}) \geq M_1$ and $\lambda_{\min}(A^*A^{*T}) \leq M_2$. Suppose the contaminated distribution $P$ satisfies $W(P,P_{\etastar}) \leq \epsilon$, and $X_i \stackrel{i.i.d.}{\sim} P$. Consider the function class $\mathcal{D}=\mathcal{D}(B, L)$, with $B = \frac{1}{M_1}$. Assume $2^L\cdot \frac{p}{n} + \epsilon^2 \leq c$ for some sufficiently small constant $c$. Then the location estimator $\widehat{\theta}$ from the W-GAN model~\eqref{model} satisfies
$$ \|\widehat{\theta} - \thetastar\|_2 \lesssim  2^L \sqrt{\frac{p}{n}} \vee \epsilon, $$
with probability at least $1 - e^{-(p+n\epsilon^2)}$, where the constant prefactor depends on $B$ and $M_2$.
\end{theorem}

The proof of Theorem~\ref{EllipLoc} is contained in Appendix~\ref{AppEllipLoc}. Note that the conclusion is of the same order as the upper bound for Gaussian location estimation in Theorem~\ref{LocUpper}.

\begin{remark}
Note that the result in Theorem~\ref{EllipLoc} is not directly implied by Remark~\ref{rmkLoc1}, since the elliptical distributions considered in this section are parametrized by $(\theta^*, h^*, A^*)$ rather than a single location parameter. Thus, we need to do some more work before using the arguments suggested by Remark~\ref{rmkLoc1}.
\end{remark}

\begin{remark}
Note the generator structure for elliptical distributions is more complex than in the Gaussian case, due to the additional parameters $h^*$ and $A^*$. In practice, we typically need to assume a parametric form of the distribution of $\xi$, i.e., $\xi = \xi(\phi)$. Then our generator contains learnable parameters corresponding to $\theta$, $A$, and $\phi$, and outputs $\widehat{\theta}+\xi\widehat{A}U$. We estimate the parameters simultaneously.
\end{remark}

%%%%%%%%%%%%%%%%%%%%%   Section 6,  Covariance matrix estimation     %%%%%%

\section{Covariance matrix estimation}
\label{sec:cov}

In this section, we consider covariance matrix estimation problems. We provide upper and lower bounds for the W-GAN based-estimator and show that the estimator achieves the minimax risk under Gaussian covariance matrix estimation. We also consider covariance matrix estimation with banded and sparse structures. Finally, we consider covariance matrix estimation for elliptical distributions.

%%%%%

\subsection{Gaussian distributions}
\label{SecCovGauss}

In Gaussian covariance matrix estimation~\citep{covlower1}, the true data generating distribution is a Gaussian distribution with known mean and unknown covariance. We would like to estimate the covariance matrix using i.i.d.\ samples from the distribution. Under the Wasserstein contamination model, we assume the true distribution is $P_\Sigma = \mathcal{N}(0,\Sigma)$, where $\Sigma$ belongs to the parameter space 
\begin{equation}
\Theta = \{\Sigma \succeq 0 : \lambda_{\max}(\Sigma) \leq M_2,  \lambda_{\min}(\Sigma) \geq M_1\}. \label{space:covGauss}
\end{equation}

We first provide an upper bound for our W-GAN-based estimator:

\begin{theorem}
\label{CovUpper}
Suppose $P_{\Sigmastar} = \mathcal{N}(0, \Sigmastar)$, where $\Sigmastar$ lies in the parameter space~\eqref{space:covGauss}. Suppose the contaminated distribution $P$ satisfies $W(P,P_{\Sigmastar}) \leq \epsilon$, and $X_i \stackrel{i.i.d.}{\sim} P$. Consider the function class $\mathcal{D}=\mathcal{D}(B, L)$, with $B=\frac{1}{M_1}$. Assume $2^L\cdot \frac{p}{n} + \epsilon^2 \leq c$ for some sufficiently small constant $c$. Then the estimator $\widehat{\Sigma}$ from the W-GAN model~\eqref{model} satisfies
$$\|\widehat{\Sigma}-\Sigmastar\|_2 \lesssim 2^L \sqrt{\frac{p}{n}} \vee \epsilon,$$
with probability at least $1 - e^{-(p+n\epsilon^2)}$, where the constant prefactor depends on $M_1$ and $M_2$.
\end{theorem}

The proof of Theorem~\ref{CovUpper} is contained in Appendix~\ref{AppCovUpper}.

We then provide a lower bound for Gaussian covariance matrix estimation under the Wasserstein contamination model:

\begin{theorem}
\label{CovLower}
Assume the true distribution is $P_{\Sigmastar} = \mathcal{N}(0, \Sigmastar)$, where $\Sigmastar$ lies in the parameter space~\eqref{space:covGauss}. Suppose the contaminated distribution $P$ satisfies $W(P,P_{\Sigmastar}) \leq \epsilon$, and $X_i \stackrel{i.i.d.}{\sim} P$. Then there exist constants $C, c>0$ such that
$$ \inf_{\widehat{\Sigma}} \sup_{\Sigmastar \in \Theta, P:W(P_{\Sigmastar}, P)\leq \epsilon} \mathbb{P}\left(\|\Sigmastar - \widehat{\Sigma}\|_2 \geq C\left(\sqrt{\frac{p}{n}} \vee \epsilon\right)\right) \geq c. $$
\end{theorem}

The proof of Theorem~\ref{CovLower} is contained in Appendix~\ref{AppCovLower}.

\begin{remark}
Comparing the bounds of Theorem~\ref{CovUpper} and~\ref{CovLower}, we see that the W-GAN-based estimator is minimax optimal. Under Huber's contamination model, the minimax rate is also $\Theta\left(\frac{p}{n} \vee \epsilon^2\right)$ and can be achieved using a depth-based estimator~\citep{chen_gao_ren_2018}.
%Our work shows that under Wasserstein contamination model, the minimax rate is the same. And the proposed W-GAN based estimator can also achieve the optimal rate.
\end{remark}

%%%%%

\subsection{Covariance matrices with special structure}
\label{SecCovSpecial}

We now consider the problem of Gaussian covariance matrix estimation when the true covariance matrix is known to have a certain structure \citep{collins_wiens_1985, cai2012optimal}. We consider two classes of covariance matrices, which are studied in \cite{chen_gao_ren_2018}: (i) banded matrices with bandwidth $k$, and (ii) sparse matrices containing a $k \times k$ nonzero submatrix with diagonal elements. More specifically, denote
$$ \mathcal{F}_1(k) = \{\Sigma = (\sigma_{ij}) \succeq 0: \sigma_{ij}=0 \text{ if } |i-j|>k\}$$
and define
\begin{align*}
\mathcal{G}(k) & = \cup_{|I|\leq k}\{G=(g_{ij}): g_{ij} = 0 \text{ if } i \notin I \text{ or } j \notin I \}, \\
\mathcal{F}_2(k) & = \{\Sigma \succeq 0: \Sigma - \text{diag}(\Sigma) \in \mathcal{G}(k) \}.
\end{align*}

The corresponding parameter spaces we consider are defined by
\begin{align}
\Theta_1(k) & = \left\{\Sigma \in \mathbb{R}^{p\times p}: \Sigma \in \mathcal{F}_1(k), M_1 \leq \|\Sigma\|_2 \leq M_2, k<\frac{p}{2}\right\}, \label{space:covSparse1} \\
\Theta_2(k) & = \left\{\Sigma \in \mathbb{R}^{p\times p}: \Sigma \in \mathcal{F}_2(k), M_1 \leq \|\Sigma\|_2 \leq M_2, k<\frac{p}{2}\right\}. \label{space:covSparse2}
\end{align}

We first consider an upper bound for the W-GAN-based estimator. Although Theorem~\ref{CovUpper} provides a valid upper bound, we may obtain a tighter upper bound using different neural network estimators that reflect the known structure of the parameter space. Define 
$$\mathcal{U}_1(B, 2k) = \bigcup_{l=1}^{p+1-2k}\{u = (u_i): \|u\|_2 \leq B, u_i = 0 \text{ if } i\notin [l, l+2k-1]\}, $$
so $\mathcal{U}_1(B, 2k)$ is the set of the vectors with $2k$ contiguous nonzero elements and 2-norm bounded by $B$. 
%$$ \mathcal{U}_2(k) = \cup_{I\subseteq [p], |I| = 2k} \{u=(u_i): \|u\|_2 = 1, u_i = 0 \text{ if } i \notin I\}, $$the set of $2k$-sparse vectors with 2-norm bounded by 1.
We define the corresponding new first layer
\begin{align}
\mathcal{D}^{(1)} & = \{f(x)=\sigma(w^T x + b): w\in \mathbb{R}^p, w\in \mathcal{U}_1(B, 2k), b\in \mathbb{R}\}. \label{func6} 
%
%\mathcal{D}_2^{(1)} & = \{f(x)=\sigma(w^T x + b): w\in \mathbb{R}^p, w\in \mathcal{U}_2(k), b\in \mathbb{R}\}. \label{func7}
\end{align}
The remaining layers are defined as before, leading to a function class which we call $\mathcal{D}_1(B,L,2k)$. We have the following bounds for the resulting W-GAN-based estimators:

\begin{theorem}
\label{SparseCovUpper1}
Assume the true distribution is $P_{\Sigmastar} = \mathcal{N}(0, \Sigmastar)$, where $\Sigmastar$ lies in the parameter space~\eqref{space:covSparse1}. Suppose the contaminated distribution $P$ satisfies $W(P,P_{\Sigmastar}) \leq \epsilon$, and $X_i \stackrel{i.i.d.}{\sim} P$. Consider the function class $\mathcal{D}=\mathcal{D}_1(B, L, 2k)$, with $B=\frac{1}{M_1}$. Assume $2^L\cdot \frac{\max(2k, \log p + 2k\log(\log{p}/(2k))}{n} + \epsilon^2 \leq c$ for some sufficiently small constant $c$. Then the estimator $\widehat{\Sigma}$ from the W-GAN model~\eqref{model} satisfies
$$\|\widehat{\Sigma}-\Sigmastar \|_2 \lesssim 2^L \sqrt{\frac{\max(2k, \log p + 2k\log(\log{p}/(2k))}{n}} \vee\epsilon,$$
with probability at least $1 - e^{-(p+n\epsilon^2)}$, and the constant only depends on $M_1$ and $M_2$.
\end{theorem}

The proof of Theorem~\ref{SparseCovUpper1} is contained in Appendix~\ref{AppSparseCovUpper1}.

\begin{remark}
Note that when $k > \frac{\log p}{2}$, we have $\max\left\{2k, \log p + 2k\log\left(\frac{\log p}{2k}\right)\right\} = 2k$, so the upper bound simplifies to $O\left(\sqrt{\frac{k}{n}} \vee \epsilon\right)$. When $k \geq \frac{p}{2}$, the first layer~\eqref{func6} in the neural network model is actually the layer~\eqref{func1}, so Theorem~\ref{CovUpper} implies the upper bound $O\left(\sqrt{\frac{p}{n}} \vee \epsilon\right)$. This is consistent with the result of Theorem~\ref{SparseCovUpper1}, where $\max\left\{2k, \log p + 2k\log(\log{p}/(2k)\right\} \geq p$, which gives the same bound $O(\sqrt{\frac{p}{n}} \vee \epsilon)$.
\end{remark}

\begin{theorem}
\label{SparseCovUpper2}
Assume the true distribution is $P_{\Sigmastar} = \mathcal{N}(0, \Sigmastar)$, where $\Sigmastar$ lies in the parameter space~\eqref{space:covSparse2}. Suppose the contaminated distribution $P$ satisfies $W(P,P_{\Sigmastar}) \leq \epsilon$, and $X_i \stackrel{i.i.d.}{\sim} P$. Consider the function class $\mathcal{D}=\mathcal{D}_{sp}(B, L, 2k)$, with $B=\frac{1}{M_1}$. Assume $2^L\cdot \frac{2k+2k\log\frac{p}{2k}}{n} + \epsilon^2 \leq c$ for some sufficiently small constant $c$. Then the estimator $\widehat{\Sigma}$ from the W-GAN model~\eqref{model} satisfies
$$\|\widehat{\Sigma}-\Sigmastar \|_2 \lesssim 2^L \sqrt{\frac{2k+2k\log\frac{p}{2k}}{n}} \vee\epsilon,$$
with probability at least $1 - e^{-(p+n\epsilon^2)}$, and the constant only depends on $M_1$ and $M_2$.
\end{theorem}

The proof of Theorem~\ref{SparseCovUpper2} is contained in Appendix~\ref{AppSparseCovUpper2}.

\begin{remark}
Again, when $k \geq \frac{p}{2}$, we obtain the same bound as in Theorem~\ref{CovUpper}, which is $O\left(\sqrt{\frac{p}{n}} \vee \epsilon\right)$.
\end{remark}

Next, we provide lower bounds for the two covariance matrix estimation problems above. The derivations follow a similar argument as before, via a modulus of continuity argument. Proofs are contained in Appendix~\ref{AppSparseCovLower}.

%For covariance matrix class $\mathcal{F}_1(k)$ and $\mathcal{F}_2(k)$, we can still use same $U$ and $V$ as in above proof. And only change in the lower bound is $\mathcal{M}(0)$, which is given by existing literature \citep{covlower1, covlower2}.

\begin{theorem}
\label{SparseCovLower}
Assume the true distribution is $P_{\Sigmastar} = \mathcal{N}(0, \Sigmastar)$, where $\Sigmastar$ lies in the parameter space~\eqref{space:covSparse1}. Suppose the contaminated distribution $P$ satisfies $W(P,P_{\Sigmastar}) \leq \epsilon$, and $X_i \stackrel{i.i.d.}{\sim} P$. Then there exist constants $C_1, c_1>0$ such that
$$ \inf_{\widehat{\Sigma}} \sup_{\Sigmastar \in \mathcal{F}_1(k), P:W(P_{\Sigmastar},P)\leq \epsilon} \mathbb{P}\left(\|\Sigmastar - \widehat{\Sigma}\|_2 \geq C_1\left(\sqrt{\frac{k+\log p}{n}} \vee \epsilon\right)\right) \geq c_1.$$
For the parameter space~\eqref{space:covSparse2}, there exist constants $C_2, c_2>0$ such that
$$ \inf_{\widehat{\Sigma}} \sup_{\Sigmastar \in \mathcal{F}_2(k), P:W(P_{\Sigmastar}, P)\leq \epsilon} \mathbb{P}\left(\|\Sigmastar - \widehat{\Sigma}\|_2 \geq C_2\left(\sqrt{\frac{k+k\log\frac{p}{k}}{n}} \vee \epsilon\right)\right) \geq c_2.$$
\end{theorem}

%The proof of Theorem~\ref{SparseCovLower} is contained in Appendix~\ref{AppSparseCovLower}.

\begin{remark}
For covariance matrices with banded structure, when $k < \frac{\log p}{2}$, there is a small gap between the upper and lower bounds, caused by the term $2k\log\left(\frac{\log p}{2k}\right)$ in the upper bound. However, this term is typically small. When $\frac{\log p}{2} \leq k < \frac{p}{2}$, both the upper and lower bounds become $\Theta\left(\sqrt{\frac{k}{n}} \vee \epsilon\right)$, which match each other. When $\frac{p}{2} \leq k$, both the upper and lower bounds become $\Theta\left(\sqrt{\frac{p}{n}} \vee \epsilon\right)$, which again match each other.

For sparse covariance matrices, when $k < \frac{p}{2}$, the bounds in Theorems~\ref{SparseCovUpper2} and~\ref{SparseCovLower} match each other. When $\frac{p}{2} \leq k$, both upper and lower bound become $O(\sqrt{\frac{p}{n}} \vee \epsilon)$, thus again, match each other.
\end{remark}

\begin{remark}
Under Huber's contamination model, when $L(\Sigma, \widehat{\Sigma}) = \| \Sigma - \widehat{\Sigma}\|_2^2$, the minimax rate for banded covariance matrix estimation is $\Theta\left(\frac{k+\log p}{n} \vee \epsilon^2\right)$. For sparse covariance matrices, the minimax rate is $\Theta\left(\frac{k+k\log\frac{p}{k}}{n} \vee \epsilon^2\right)$. Both rates can be achieved using matrix-depth-based estimators~\citep{chen_gao_ren_2018}.
%Our result generalizes their result to another setting, Wasserstein contamination model and we can see we prove the same minimax rate and it can be attained by W-GAN based estimator.
\end{remark}

\subsection{Elliptical distributions}

In fact, the results in Section~\ref{SecCovGauss} can be extended to covariance matrix estimation for elliptical distributions. Suppose $X$ has the stochastic representation $X = \xi A^*U$, where $\xi$ is a nonnegative random variable, $A^* \in \mathbb{R}^{p \times p}$ is a deterministic matrix, $U \in \mathbb{R}^p$ is uniformly distributed on the unit sphere, and $\xi$ is independent of $U$. The matrix $\Sigmastar = A^*A^{*T}$ is called the scatter matrix of $X$. Note that for any vector $u$, the distribution of $u^T\xi U$ does not depend on $u$ because of the symmetry of $U$. Let $h^*(t)$ denote the density of $u^T\xi U$. Since there is a one-to-one correspondence between $h^*(t)$ and distribution of $\xi$, the distribution of $X = \xi A^*U$ is determined by $h^*$ and $A^*$, and we denote the distribution by $\mathcal{E}(h^*, A^*)$.

%Note $h^*$ and $A^*$ are not identifiable since for any $a > 0$, $\xi A^*U = \frac{\xi}{a} (aA^*)U$. Thus we add the constraint that 
In order to avoid identifiability issues, we impose a slightly different normalization condition than in the case of equation~\eqref{EqnRestrict}. We define the function class
$$ \mathcal{H} = \left\{h^*(\cdot) : h^*(t)\geq 0, \int h^*(t) = 1, \int R(|t|)h^*(t) dt = \int R(|t|) \phi(t) dt \right\},$$
where $\phi(t)$ is the density of the standard normal distribution and $R(t) = \min\{|t|, 1\}$.
% is defined as 
%\begin{equation*}
%R(t) =  
%\begin{cases}
%|t| & \text{if } t \leq 1. \\
%1, & \text{otherwise}.
%\end{cases}
%\end{equation*}
This idea follows the parametrization used in~\cite{Gao2}. We also define ramp($u$) = $2\max\{\min\{u+1/2,1\},0\}$, and define the layer
\begin{equation}
\mathcal{D}^{(1)} = \left\{f(x)=\text{ramp}(w^T x + b): w\in \mathbb{R}^p, \|w\|_2 \leq B,  b\in \mathbb{R}\right\}. \label{funcellip}
\end{equation}
We denote the function class defined by equations~\eqref{funcellip}, \eqref{func2}, and~\eqref{func3} by $\mathcal{D}_e(B, L)$.

For estimation of $\Sigmastar$, we have the following theorem:

\begin{theorem}
\label{EllipCov}
Assume the true distribution $P_{h^*,A^*} = \mathcal{E}(h^*, A^*)$, where $\Sigmastar=A^*A^{*T}$ lies in the parameter space~\eqref{space:covGauss}. Suppose the contaminated distribution $P$ satisfies $W(P,P_{h^*, A^*}) \leq \epsilon$, and $X_i \stackrel{i.i.d.}{\sim} P$. Consider the function class $\mathcal{D}=\mathcal{D}_e(B, L)$, with $B=\frac{1}{M_1}$. Assume $2^L\cdot \frac{p}{n} + \epsilon^2 \leq c$ for some sufficiently small constant $c$. Then the covariance matrix estimator $\widehat{\Sigma}$ from the W-GAN model~\eqref{model} satisfies 
$$\|\widehat{\Sigma}-\Sigmastar \|_2 \lesssim 2^L\sqrt{\frac{p}{n}} \vee \epsilon,$$
with probability at least $1 - e^{-(p+n\epsilon^2)}$, where the constant prefactor depends on $M_1$ and $M_2$.
\end{theorem}

The proof of Theorem~\ref{EllipCov} is contained in Appendix~\ref{AppEllipCov}. We can see that the structure of the covariance matrix does not affect the proof; thus, for elliptical distributions with banded or sparse covariance matrices, if we use the corresponding estimators described in Section~\ref{SecCovSpecial}, we can obtain the same rates as in Theorems~\ref{SparseCovUpper1} and~\ref{SparseCovUpper2}.

\begin{remark}
Note that we also constrain the estimate to be in the true parameter space, so that $\widehat{\Sigma}$ lies in the space~\eqref{space:covGauss} and $\widehat{h}$ satisfies $\int R(|t|)\widehat{h}(t) dt = \int R(|t|) \phi(t) dt$. We can typically realize this via a projection method, i.e., after updating the estimate of $\widehat{\Sigma}$ at each step, we project it into the true parameter space.
\end{remark}

%%%%%%%%%%%%%%%%%%%%%   Section 7, Linear Regression     %%%%%%

\section{Linear regression}
\label{sec:reg}
%Or simply put, for the observed $y_i$'s, can we say they are observations of a same distribution?(actually not). So that we can set up model as $\min W(P_n, P_{\widehat{\beta}})$.
%It seems that even if we consider conditional distribution, we only get 1 observation for each conditional distribution conditioning on $X_i$. So if we model as $\min W(P|X_i, P_{\widehat{\beta}}|X_i)$, there seems to be n models but not 1 model with n observations. This causes the problem.

In the linear regression model, for a random vector $(X,Y)$, where $X\in \mathbb{R}^p$ and $Y\in \mathbb{R}$, we assume that $X \sim \mathcal{N}(0,\Sigmastar)$ and the conditional distribution is $P_{Y|X} = \mathcal{N}(X^T \betastar, 1)$. We denote the joint distribution of $(X,Y)$ by $P_{\betastar}$ (the distribution also depends on $\Sigmastar$, but for simplicity of notation, we omit $\Sigmastar$). Under the Wasserstein contamination model, we wish to estimate $\betastar$ using i.i.d.\ samples from a distribution $P$ satisfying $W(P, P_{\betastar}) \leq \epsilon$.

Loss functions of interest for linear regression can take several different forms. One common loss function is the $\ell_2$-error between the true and estimated parameters, i.e., $\|\betastar - \widehat{\beta}\|_2$. An alternative loss function is the prediction error, $\E_{X \sim \mathcal{N}(0, \Sigma)} (X^T\betastar - X^T\widehat{\beta})^2$. Another version of the prediction error is $\E_{(X,Y)}(Y - X^T\widehat{\beta})^2$, where the expectation is taken over the true joint distribution $(X,Y)$. Our goal is to find an estimator which minimizes the worst-case risk $\max\limits_{P_{\betastar}, P} \E L(\betastar, \widehat{\beta})$, where the loss function takes one of the above three forms and the expectation is taken over the observed data set. We first provide upper and lower bounds for our proposed estimator when $L$ is the squared $\ell_2$-error (Theorems~\ref{RegUpper} and~\ref{RegLower}), showing the near-optimality of the proposed W-GAN-based estimator. We also provide a lower bound in the case of prediction error, $L(\betastar, \widehat{\beta}) =\E_X(X^T\betastar - X^T\widehat{\beta})^2$.

\subsection{Upper bound}
\label{SecRegUpper}

We first provide an upper bound for the W-GAN based estimator under the above regression model setup. We consider the parameter space
\begin{equation}
\Theta = \{\betastar \in \mathbb{R}^p, \Sigmastar \in \mathbb{R}^{p\times p}: \|\betastar\|_2\leq B_1, B_2 \leq \|\Sigmastar\|_2 \leq B_3\}. \label{space:reg}
\end{equation}
The proof the following result is contained in Appendix~\ref{AppRegUpper}:

\begin{theorem}
\label{RegUpper}
Assume the true distribution is $P_{\betastar}$, where $\betastar$ lies in the parameter space~\eqref{space:reg}. Suppose the contaminated distribution $P$ satisfies $W(P,P_{\betastar}) \leq \epsilon$, and $(X_i, Y_i) \stackrel{i.i.d.}{\sim} P$. Consider the function class $\mathcal{D}=\mathcal{D}(B, L)$, with $B = B_1+ 1$. Assume $2^L\cdot \frac{p}{n} + \epsilon^2 \leq c$ for some sufficiently small constant $c$. Then the estimator $\widehat{\beta}$ from the W-GAN model~\eqref{model} satisfies
$$\|\beta - \widehat{\beta}\|_2^2 \lesssim 2^L \sqrt{\frac{p}{n}} \vee \epsilon, $$
with probability at least $1 - e^{-(p+n\epsilon^2)}$, where the constant only depends on $B_1, B_2$, and $B_3$.
\end{theorem}

\begin{remark}
\cite{Zhu1} also considered regression under the Wasserstein-1 contamination model, but they focused on the prediction error loss, $\E_{p^*}(Y-X^T\widehat{\beta})^2$, where $p^*$ is the true distribution of $(X,Y)$. Under certain conditions, their proposed algorithm has an upper bound of $O\left(\left(\epsilon \vee \frac{p}{n}\right)^{1-1/(k-1)}\right)$, for a true distribution $p^*$ having finite $k^{\text{th}}$ moments.
\end{remark}

\subsection{Lower bound on estimation error}
\label{SecRegLower1}

Assume $X\sim \mathcal{N}(0,\sigma^2I_p)$ and $Y = X^T\beta+z$, where $z\sim \mathcal{N}(0,1)$ is independent of $X$.
%Consider the loss function $L(\beta_1, \beta_2) = \|\beta_1 - \beta_2\|_2^2$.
We consider the parameter space
\begin{equation}
\label{EqnRegLowerParam}
\Theta = \left\{\betastar \in \mathbb{R}^p: \|\betastar\|_2\leq B_1, B_1 > \frac{\sqrt{p}}{2}\right\}.
\end{equation}
Let $P_{\betastar}$ denote the joint distribution of $(X,Y)$.
The following result provides a lower bound on the minimax risk in terms of the $\ell_2$-error:

\begin{theorem}
\label{RegLower}
Assume the true distribution is $P_{\betastar}$, where $\betastar$ lies in the parameter space~\eqref{EqnRegLowerParam}. Suppose the perturbed distribution $P$ satisfies $W(P, P_{\betastar})\leq \epsilon$, and $(X_i, Y_i) \stackrel{i.i.d.}{\sim} P$. Then there exist constants $C, c>0$ such that
$$ \inf_{\widehat{\beta}} \sup_{\betastar \in \Theta, P:W(P_{\betastar}, P)\leq \epsilon} \mathbb{P}\left(\|\betastar - \widehat{\beta}\|_2 \geq C\left(\sqrt{\frac{p}{n}} \vee \frac{\sqrt{\epsilon}}{\sigma}\right)\right) \geq c. $$
\end{theorem}

The proof of Theorem~\ref{RegLower} is contained in Appendix~\ref{AppRegLower}. Comparing the upper and lower bounds, we can see that they indeed match in terms of the order of $\epsilon$, although for the order of $n$, the bounds do not match. We conjecture that the upper bound is not tight, and either requires a more careful analysis or a different class of estimators.

%%%%%

\subsection{Lower bound on prediction error}
%Assume $X\sim \mathcal{N}(0,\Sigma)$, $Y = X^T\beta+z$ where $z\sim \mathcal{N}(0,1)$ and independent of $X$. Then $(X,Y)\sim \mathcal{N}(0, \Gamma)$, where 
%\begin{equation*}
%\Gamma = 
%\begin{pmatrix}
%      \Sigma & \Sigma\beta \\
%      \beta^T\Sigma & \beta^T\Sigma\beta + 1
%\end{pmatrix}      
%\end{equation*}
If we instead consider the loss function $L(\beta_1, \beta_2) = \E_{X}(X^T\beta_1 - X^T\beta_2)^2,$ then by similar arguments as in Section~\ref{SecRegLower1}, we may derive the following result:

\begin{theorem}
\label{RegLower2}
Assume the true distribution is $P_{\betastar}$, where $\betastar$ lies in the parameter space~\eqref{EqnRegLowerParam}. Suppose the perturbed distribution $P$ satisfies $W(P, P_{\betastar})\leq \epsilon$, and $(X_i, Y_i) \stackrel{i.i.d.}{\sim} P$. Then there exist constants $C, c>0$ such that
$$ \inf_{\widehat{\beta}} \sup_{\beta \in \Theta, P:W(P_\beta,P)\leq \epsilon} \mathbb{P}\left(\E_X\|X^T\beta - X^T\widehat{\beta}\|_2^2 \geq C\left(\frac{p}{n} \vee \epsilon\right)\right) \geq c. $$
\end{theorem}

The proof of Theorem~\ref{RegLower2} is contained in Appendix~\ref{AppRegLower2}.
Theorem~\ref{RegLower2} provides a lower bound on the minimax \emph{prediction} rather than estimation error, as provided by Theorem~\ref{RegLower}. Since the prediction error loss function involves taking an expectation over $X$, the $\sigma$ factor is canceled out. The common term $\frac{p}{n}$ in both bounds arises from the minimax rate of linear regression estimation without contamination.

%\textbf{Remark}. For linear regression, we can consider at least 3 different loss functions, $\|\beta - \widehat{\beta}\|_2^2$, $\E_X(X^T\beta - X^T\widehat{\beta})^2$ and $\E_{(X,Y)}(Y - X^T\widehat{\beta})^2$. And each loss is is random variable, where randomness comes from $\widehat{\beta}$, essentially, from observed data $(X_i, Y_i)$. Thus for each loss, we have a probabilistic statement and expectation statement.

%However, our upper bound in section 7.2 is not a standard probabilistic statement, since the standard one would consider bounding $\E_X(X^T\beta - X^T\widehat{\beta})^2$ in a probabilistic statement, not $(X^T\beta - X^T\widehat{\beta})^2$. This is same for theorem in section 7.5, in there, they also considered $|X^T\beta - X^T\widehat{\beta}|$.

%With a larger $\sigma$, we need $\beta_1$ and $\beta_2$ to be much closer so that the two distributions satisfies the Wasserstein distance constraint, thus the modulus of continuity is small(the lower bound is inverse proportional to $\sigma$. In other words, currently, the lower we can get via modulus of continuity is inverse proportional to $\sigma$.

%Change $\beta_2$ to $\frac{\sqrt{\epsilon}}{\sigma^2} u$, same result as before. \\
%Change $\beta_2$ to $\sqrt{\epsilon}{\sigma^2}$, cannot bound Wasserstein distance. \\
%Change first diagonal of $\sigma^2I$ to $\sigma_1^2$, not working. \\

\section{Numerical experiments}
\label{sec:num}

In this section, we provide empirical results on the performance of the W-GAN-based estimator~\eqref{model}.
%our W-GAN based estimator is defined as
%$$\widehat{\theta}=\argmin_{\eta}\sup_{f \in \mathcal{D}} \E_{P_n}f(X) - \E_{P_\eta}f(X).$$
%Function $f(x)$ is typically called the discriminator, and another function $g(z)$, which is used to generate the random variable $g(Z)$, which has distribution $P_\eta$, is called the generator, here $Z$ is a random variable following a known prior distribution.
%As we have seen above, the function class $\mathcal{D}$ is implemented as neural network models, with model structure changing for different estimation problems.
Additional details regarding implementation are provided in Appendix~\ref{AppImplementation}, and tables containining full numerical results are contained in Appendix~\ref{AppExp}.

%%%%%

\subsection{Location estimation}
\label{SecLocSetup}

For the discriminator, as stated in Theorem~\ref{LocUpper}, we follow the function class defined by equations~\eqref{func1}, \eqref{func2}, and~\eqref{func3}. The input is a $p$-dimensional vector; the first layer is a dense layer with $\frac{p}{2}$ output units and sigmoid activation function; the second layer is a dense layer with $\frac{p}{4}$ output units and ReLU activation function; and the output layer is a dense layer without activation. For the generator, the input is a random sample $Z \sim \mathcal{N}(0,I_p)$ and the output is $\eta^{(t)}+Z$, where $\eta^{(t)}$ is the estimate at time $t$.

The true distribution is $\mathcal{N}(0,I_p)$, and we consider three different perturbation models:\begin{itemize}
\item Model 1: The perturbed distribution is $0.9\mathcal{N}(\theta,I_p) + 0.1Q$, where $Q$ is the standard Cauchy distribution.
\item Model 2: The perturbed distribution is $0.9\mathcal{N}(\theta,I_p) + 0.1\mathcal{N}(2*\mathbf{1_p}, I_p)$.
\item Model 3: The perturbed distribution is $0.9\mathcal{N}(\theta,I_p) + 0.1Q$, where $Q$ is the standard Gumbel distribution.
\end{itemize}

%For comparison, we compare the loss with the empirical mean estimator. 
%For $n=100$, trained for 200 epochs, second hidden layer using 2-norm constraint, step size 0.01. \\
%For $n=1024$, trained for 200 epochs, second hidden layer using 1-norm constraint, step size 0.01. \\
%For $n=1024, p=80$, step size is 0.005. \\
%For Model 2, $n=4096, p=40,80$ and $n=10000$, use half hidden units. \\
%For $n=1024, p=40$, seed is 5. \\
%For $n=4096, p=10, 20$, seed is 5. \\

As seen in Table~\ref{TabLocation}, the W-GAN-based estimator consistently outperforms the simple sample mean estimator. Among the three perturbation models, the Cauchy distribution has the heaviest tails and generates more outliers. We can see that the sample mean has a much larger error when $n$ and $p$ increase. However, the error of the W-GAN-based estimator is much smaller. Furthermore, the error of the W-GAN-based estimator decreases with the sample size, while the error of the sample mean  does not.
%This again shows the robustness of our estimator.

%%%%%

\subsection{Covariance matrix estimation}

For the discriminator, as stated in Theorem~\ref{CovUpper}, we again follow the function class defined by equations~\eqref{func1}, \eqref{func2}, and~\eqref{func3}, as described in Section~\ref{SecLocSetup}.
%The input is a $p$-dimensional vector; the first layer is a dense layer with $\frac{p}{2}$ output units and sigmoid activation function; the second layer is a dense layer with $\frac{p}{4}$ output units and ReLU activation function; and the output layer is a dense layer without activation.
For the generator, the input is random sample $Z \sim \mathcal{N}(0,I_p)$ and the output is $A^{(t)}Z$, where $A^{(t)}$ is the weight parameter and $A^{(t)}{A^{(t)}}^{T}$ is the estimate of the covariance matrix at time $t$.

The true distribution is $\mathcal{N}(0, I_p)$, and the three perturbations models are the same as in the location estimation setting.
As seen in Table~\ref{TabCovariance}, our W-GAN-based estimator consistently outperforms the simple sample covariance estimator and is very robust to data perturbations. When adding Cauchy perturbations, the sample covariance estimator can be fairly unstable, yet the W-GAN based estimator still gives reasonable estimations. Note that when $n>1000$, the decrease in spectral norm loss is not significant; however, this may be due to the effect of the $\epsilon$ term in the bound $O\left(\sqrt{\frac{p}{n}}\vee \epsilon\right)$.

%%%%%

\subsection{Regression}

For discriminator, we follow the function class defined by equations~\eqref{func4}, \eqref{func2}, and \eqref{func3}. The input is a $p$-dimensional vector; the first layer is a dense layer with $\frac{p}{2}$ output units and sigmoid activation function; the second layer is a dense layer with $\frac{p}{4}$ output units and ReLU activation function; and the output layer is a dense layer without activation. For the generator, the input is a random sample $(X, Z) \sim \mathcal{N}(0,I_{p+1})$ and the output is $(X, X^T\beta^{(t)}+Z)$, where $\beta^{(t)}$ is the weight parameter for the generator at time $t$.

The true distribution is defined by $X \sim \mathcal{N}(0, I)$, $Y|X \sim \mathcal{N}(X^T\beta, 1)$, where $\beta$ is vector $(-0.05,-0.05,0,0.05,0.05)$ repeated $\frac{p}{5}$ times. The perturbed model is $X \sim \mathcal{N}(0, I)$, $Y|X \sim 0.8\mathcal{N}(X^T\beta, 1)+0.2Q$, where $Q$ is absolute value of the standard Cauchy distribution. For comparison, we use least squares estimation.

From Table~\ref{TabRegression1}, we can see that the W-GAN based estimator significantly outperforms the OLS estimator, and the $\ell_2$-norm loss decreases with the sample size. Table~\ref{TabRegression2} shows the loss for different perturbation levels. The loss increases with $\epsilon$, as predicted by our theory.

% p=10,20, seed=4, p=40, seed=5

%%%%%

\section{Discussion and future work}
\label{sec:dis}

We have proposed a W-GAN-based estimator for robust estimation under Wasserstein contamination. Our estimator is applicable to a wide range of problems, including location estimation, covariance estimation and regression, and in many cases is seen to be minimax optimal. We have also presented promising numerical results from training W-GANs.

%How could GAN solve the robust estimation problem? Is there an intuitive way of explaining this connection between the statistical problem and the deep learning model?

%Actually, robust estimation is finding an estimator such that $ \sup_{P_\theta,Q} \|\theta-\widehat{\theta}\|^2$ is small, where the loss is squared 2-norm distance between parameters. They try to explicitly solve the problem and find such an estimator. But GANs are essentially minimizing distance between two distributions, specifically, W-GAN is minimizing Wasserstein distance between two distributions. So we can prove using W-GAN approach, we can get an estimator for which the Wasserstein distance between true and estimated distribution is small (for any true distribution and any valid perturbation), i.e. $ \sup_{P_\theta,Q} W(\theta, \widehat{\theta})$ is small. Then it's natural to guess the 2-norm distance between parameters are bounded by Wasserstein distance and actually we can prove it. Thus we can get a good estimator.

%For instance, for robust estimation problems under total variation distance perturbation, we can make use of TV-GAN. And for contaminations under other probability distances, we can also utilize corresponding GAN approach.

The connections we have drawn between robust estimation and GANs are both theoretically interesting and practically useful: Models from deep learning which are not completely characterized from a theoretical standpoint can nonetheless be leveraged to achieve the minimax rate for a variety of robust estimation problems. Furthermore, although optimal robust estimators are often computationally intractable, training GAN-based estimators is becoming easier with advances in deep learning training platforms. The framework studied in this paper has natural generalizations to other contamination models, in which the distance used to quantify perturbations would be encoded into the metric of the GAN.

We conclude by mentioning a few open questions. In the examples we have studied, we have imposed various constraints on the parameter space to derive upper bounds on our estimators; can we still obtain optimal estimators using W-GANs without including those constraints on the parameter space? In the linear regression setting, our upper and lower bounds do not quite match, so the question of whether W-GAN-based estimators are minimax optimal (in both estimation and prediction error) also remains open. It would also be interesting to study W-GAN-based estimators for sparse linear regression. We leave these questions to future work.

%%%%%

\section*{Acknowledgments}
The authors gratefully acknowledge support from NSF grant DMS-1749857.

%%%%%

%\bigskip
%\begin{center}
%{\large\bf SUPPLEMENTARY MATERIAL}
%\end{center}

%\begin{description}

%\item The supplementary appendix contains technical details of the proofs and details for the simulation results.

%\end{description}

\bibliographystyle{chicago}
\bibliography{Bibliography-MM-MC}

\newpage
\appendix

%%%%%

\section{Experimental results}

In this appendix, we provide additional implementation details for our experiments, as well as tables containing the numerical results referenced in Section~\ref{sec:num}.

\subsection{Implementation details}
\label{AppImplementation}

We denote the weight parameters in the neural network model by $w$, and denote the functions in $\mathcal{D}$ by $f_w$, corresponding to the discriminator. The generator, which is used to obtain $P_\theta$, depends on the specific problem: For location estimation, the generator is just an addition layer, $g(Z) = \theta + Z$, where the input is $Z \sim \mathcal{N}(0, I_p)$. For covariance matrix estimation, $g(Z) = AZ$, where the input is $Z \sim \mathcal{N}(0, I_p)$ and $AA^T$ is the estimation for covariance matrix. For regression, $g(X, Z) = (X, X^T\beta + Z)$, where the input is $(X, Z) \sim \mathcal{N}(0, I_{p})$.

The training alternates between two steps: optimizing over $w$ and optimizing over $\theta$. At time step $t$, we fix $\theta^{(t)}$ and optimize over $w$ to maximize $\E_{P_n}f_w(X) - \E_{P_{\theta^{(t)}}}f_w(X)$, where the expectation over $\E_{P_{\theta^{(t)}}}$ is approximated by sample average. Then we fix $w^{(t)}$ and optimize over $\theta$ to minimize $\E_{P_n}f_{w^{(t)}}(X) - \E_{P_\theta}f_{w^{(t)}}(X)$, which is equivalent to maximizing $\E_{P_\theta}f_{w^{(t)}}(X)$. Again, the expectation over $\E_{P_\theta}$ is approximated by the sample average.

We use the RMSprop optimizer~\citep{tieleman2017divide} and choose the step size to be 0.005. We tested different step sizes and concluded that 0.005 worked best. We trained the model in a batch version, with batch size 32 for sample size 100, batch size 128 for sample size larger than 1000 and less than 5000, and batch size 256 for sample size larger than 5000. We increased the batch size with larger training sets to make full use of computing resources and accelerate training. 

Finally, note that we implemented the 1-norm or 2-norm constraints on weight parameters during training using a simple truncation method. After using the ordinary RMSProp optimization algorithm to update the weight parameters, we normalized the weights to satisfy the required constraints.

%%%%%

\subsection{Numerical results}
\label{AppExp}

See Tables~\ref{TabLocation}, \ref{TabCovariance}, \ref{TabRegression1}, and~\ref{TabRegression2}.

\begin{table}[h!]
\begin{center}
\begin{tabular}{|c|c|c|c|c|c|}
\hline
\multicolumn{2}{|c|}{Model} & n= 100 & n=1024 & n=4096 & n=10000 \\ \hline
&p = 10 & \makecell{0.1263(0.0472)\\0.2362} & \makecell{0.0635(0.0260)\\0.1332} &\makecell{0.0392(0.0146)\\0.6462} & \makecell{0.0435(0.0156)\\0.7495}\\
\cline{2-6}
M1&p = 20 & \makecell{0.3692(0.1095)\\16315.9303} &  \makecell{0.0853(0.0119)\\1.7714} & \makecell{0.0616(0.0112)\\21.2478}  & \makecell{0.0527(0.0183)\\2.1945}  \\ 
\cline{2-6}
&p = 40 & \makecell{0.6685(0.0949)\\20.3320} & \makecell{0.1875(0.0124)\\11.7123}  & \makecell{0.0791(0.0166)\\581.6685}  & \makecell{0.0837(0.0157)\\16.1688}  \\ 
\cline{2-6}
&p = 80 &\makecell{1.7087(0.3155)\\31.4197} & \makecell{0.30546(0.0346)\\109.0997}  &  \makecell{0.1546 (0.0246)\\21.2031} &  \makecell{0.1662(0.0332)\\32.5745} \\
\hline
&p = 10 & \makecell{0.3902(0.1349)\\0.5007}  & \makecell{0.2603(0.1920)\\0.4045}  &  \makecell{0.1773(0.1790)\\0.4230} & \makecell{0.2796(0.3861)\\0.3916}  \\ 
\cline{2-6}
M2&p = 20 & \makecell{0.6064(0.1281)\\0.7057} & \makecell{0.1135(0.0701)\\0.8483}  & \makecell{0.1213(0.0929)\\0.8730}  & \makecell{0.1189(0.0479)\\0.8591}  \\ 
\cline{2-6}
&p = 40 & \makecell{0.8871(0.1359)\\1.2490}  &  \makecell{0.1342(0.0495)\\1.4181} &  \makecell{0.1253(0.0664)\\1.1583} &  \makecell{0.0575(0.0213)\\1.6888} \\ 
\cline{2-6}
&p = 80 & \makecell{17.3555(4.5030)\\12.1336}& \makecell{0.2620(0.0318)\\3.3112}  & \makecell{0.4882(0.2159)\\2.9795}  & \makecell{0.4875(0.1633)\\3.2067}  \\
\hline
&p = 10 & \makecell{0.1556(0.0530)\\0.0846} & \makecell{0.0851(0.0529)\\0.0503}  & \makecell{0.0898(0.0747)\\0.0374}  & \makecell{0.0659(0.0503)\\0.0372}  \\ 
\cline{2-6}
M3&p = 20 &\makecell{0.3943(0.0967)\\0.2635} & \makecell{0.1233(0.0512)\\0.0846}  & \makecell{0.0985(0.0680)\\0.0849}  & \makecell{0.1367(0.0590)\\0.0749}  \\ 
\cline{2-6}
&p = 40 & \makecell{0.8033(0.0830)\\0.6013} &  \makecell{0.1639(0.0368)\\0.1326} &  \makecell{0.2169(0.0782)\\0.1529} &  \makecell{0.0959(0.0343)\\0.1408} \\ 
\cline{2-6}
&p = 80 & \makecell{2.0441(0.1171)\\1.8067} &  \makecell{ 0.3699(0.0536)\\0.3336} &  \makecell{0.1902(0.0278)\\0.2828} &  \makecell{0.2754(0.0627)\\0.2749} \\
\hline
\end{tabular}
\caption{Squared error loss, $\|\widehat{\theta}-\theta\|_3^2$ (Theorems~\ref{LocUpper} and~\ref{LocLower}). The discriminator has two hidden layers. For each $(p,n)$ combination, we repeat the W-GAN training 10 times and the upper number is the mean squared error loss (the upper number in parentheses is the standard deviation). The lower number is the squared error loss for the sample mean estimator. \label{TabLocation}}
\end{center}
\end{table}%

\begin{table}[h!]
\begin{center}
\begin{tabular}{|c|c|c|c|c|c|}
\hline
\multicolumn{2}{|c|}{Model} & n= 100 & n=1000 & n=5000 & n=10000 \\ \hline
&p = 10 & \makecell{2.6706 (0.6188)\\989.1726} & \makecell{2.5324 (0.6786)\\3331.3031} &\makecell{3.7381 (1.5220)\\1018.1120} & \makecell{33.6696 (72.5449)\\929.4809}\\
\cline{2-6}
M1&p = 20 & \makecell{3.9733 (0.5873)\\341.9525} &  \makecell{3.3892 (0.9017)\\9317.2763} & \makecell{5.3617 (1.2598)\\31848.3172}  & \makecell{5.5856 (1.1482)\\4843.0940}  \\ 
\cline{2-6}
&p = 40 & \makecell{6.0299 (0.7105)\\1992.2081} & \makecell{4.6945 (1.0452)\\907.6427}  & \makecell{6.6223 (1.2968)\\526805.2677}  & \makecell{7.2449 (1.2185)\\6988452.4528}  \\ 
\hline
&p = 10 & \makecell{3.3931 (1.3983)\\22.6888}  & \makecell{3.6228 (1.0960)\\23.4328}  &  \makecell{5.6791 (2.4001)\\21.9613} & \makecell{5.4892 (2.6689)\\22.1124}  \\ 
\cline{2-6}
M2&p = 20 & \makecell{3.0725 (0.5562)\\22.9557} & \makecell{4.0251 (0.6219)\\42.4484}  & \makecell{7.4081 (1.1769)\\42.1029}  & \makecell{7.7177 (1.3608)\\43.9374)}  \\ 
\cline{2-6}
&p = 40 & \makecell{5.5444 (0.7963)\\82.7338}  &  \makecell{4.2158 (0.5013)\\83.8858} &  \makecell{9.0508 (1.3909)\\95.6900} &  \makecell{8.1087 (0.7254)\\87.1616} \\ 
\hline
&p = 10 & \makecell{3.0510 (0.6039)\\28.0165} & \makecell{3.1097 (0.6735)\\13.6526}  & \makecell{1.8482 (0.9037)\\10.6381}  & \makecell{1.4652 (0.9515)\\10.9958}  \\ 
\cline{2-6}
M3&p = 20 &\makecell{4.4205 (0.6598)\\43.0277} & \makecell{3.3260 (1.1579)\\22.8060}  & \makecell{2.0038 (0.4412)\\18.5284}  & \makecell{2.4031 (1.0378)\\19.2928}  \\ 
\cline{2-6}
&p = 40 & \makecell{5.4583 (0.8564)\\62.7293} &  \makecell{4.2630 (0.6037)\\36.2662} &  \makecell{5.4488 (0.6776)\\35.8376} &  \makecell{5.4391 (0.6912)\\33.0079} \\ 
\hline
\end{tabular}
\caption{Spectral norm loss, $\|\Sigma - \widehat{\Sigma}\|_2$ (Theorems~\ref{CovUpper} and~\ref{CovLower}). The discriminator has two hidden layers. For each $(p,n)$ combination, we repeat the W-GAN training for 10 times and the upper number is the mean spectral norm loss (the upper number in parentheses is the standard deviation). The lower number is the spectral norm loss for the sample covariance matrix estimator. \label{TabCovariance}}
\end{center}
\end{table}%

\begin{table}[h!]
\begin{center}
\begin{tabular}{|c|c|c|c|c|}
\hline
 & n=100 & n=1000 & n=5000 & n=10000 \\ \hline
p=10 &\makecell{1.0109 (0.3597)\\1.0677} &\makecell{0.5363 (0.1814)\\0.4556} &\makecell{0.4346 (0.0861)\\0.6734} &\makecell{0.4568 (0.0654)\\1.5149} \\ \hline
p=20 &\makecell{1.1022 (0.0662)\\0.7083} &\makecell{0.6872 (0.0659)\\2.2157} &\makecell{0.4855 (0.0747)\\0.9745} &\makecell{0.5422 (0.048)\\0.6300} \\ \hline
p=40 &\makecell{1.5738 (0.138)\\1.4732} &\makecell{0.8461 (0.1064)\\0.8298} &\makecell{0.9101 (0.1157)\\2.1426} &\makecell{0.5152 (0.0573)\\0.3434} \\ \hline
p=80 &\makecell{2.7049 (0.091)\\3.7407} &\makecell{1.2745 (0.0992)\\1.2491} &\makecell{0.9946 (0.0801)\\2.3301} &\makecell{0.6947 (0.0399)\\21.4306} \\ \hline
\end{tabular}
\caption{2-norm loss, $\|\beta-\widehat{\beta}\|_2$ (Theorems~\ref{RegUpper} and~\ref{RegLower}). The discriminator has two hidden layers. For each $(p,n)$ combination, we repeat the W-GAN training for 10 times and the upper number is the mean spectral norm loss (the upper number in parentheses is the standard deviation). The lower number is the spectral norm loss for the ordinary least squares (OLS) estimator. \label{TabRegression1}}
\end{center}
\end{table} %

\begin{table}[h!]
\begin{center}
\begin{tabular}{|c|c|c|c|c|}
\hline
& $\epsilon$=0.05 & $\epsilon$=0.10  & $\epsilon$=0.20 & $\epsilon$=0.50 \\ \hline
p=10 & \makecell{0.4579 (0.1404)\\ 0.1404} &\makecell{0.5122 (0.1762)\\0.2086} &\makecell{0.5273 (0.1398)\\0.9581}  &\makecell{0.5623 (0.1350)\\1.8899} \\ \hline
% p=20 &\makecell{0.6639 (0.1103)\\0.1957\\0.5149 (0.0651)\\0.2357} &\makecell{0.7183 (0.1371)\\0.3347\\0.4819 (0.0872)\\0.9459} &\makecell{0.6488 (0.0896)\\0.3368\\0.4855 (0.0747)\\0.9745} &\makecell{0.6613 (0.1143)\\0.4200\\0.4923 (0.0570)\\1.0896} \\ \hline % seed = 2
p=20 & \makecell{0.5268 (0.0910)\\1.6259} & \makecell{0.5248 (0.0918)\\3.2439}  &\makecell{0.5626 (0.1086)\\0.6579} & \makecell{0.6221 (0.1319)\\2.9520} \\ \hline % seed =3
p=40 & \makecell{0.6182 (0.0693)\\0.3524} & \makecell{0.6373 (0.0524)\\0.4137} &\makecell{0.6957 (0.0954)\\1.1675} & \makecell{0.7033 (0.0868)\\2.1114} \\ \hline
p=80 & \makecell{0.7589 (0.0757)\\0.6753} & \makecell{0.7628 (0.0646)\\1.4139} &\makecell{0.7970 (0.0793)\\52.0010} & \makecell{0.8442 (0.1046)\\3.2878} \\ \hline
\end{tabular}
\caption{2-norm loss, $\|\beta-\widehat{\beta}\|_2$ (Theorems~\ref{RegUpper} and~\ref{RegLower}). Here, the sample size is $n=4096$. For each $(p,\epsilon)$ combination, we repeat the W-GAN training for 10 times and the upper number is the mean spectral norm loss and the upper number in the bracket is the standard deviation. The lower number is the spectral norm loss for the ordinary least squares (OLS) estimator. \label{TabRegression2}}
\end{center}
\end{table}

%%%%%

\section{GANs and GAN-based estimators}
\label{AppGAN}

In this appendix, we provide a brief review of GANs. Recall that the original GAN model solves an minimax optimization problem, and is formulated as follows \citep{NIPS2014_5423}:
\begin{equation*}
\inf_G \sup_D \left\{\E_{X \sim P_{\thetastar}} \log D(X) + \E_Z \log(1-D(G(Z)))\right\},
\end{equation*}
where the function $G$ is known as the generator, the function $D$ is known as the discriminator, $P_{\thetastar}$ is the true distribution, and the second expectation is taken over the known prior distribution of $Z$. The function $D$ estimates the probability that a sample is drawn from the true distribution $P_{\thetastar}$ rather than the generator distribution $G(\cdot)$.

Suppose we fix $G$ and consider the inner supremum. Let $Q$ denote the distribution of $G(Z)$. Note that this $Q$ depends on $G$, but we do not explicitly show it in the notation, for simplicity. Let $p(x)$ and $q(x)$ denote the probability density functions of $P_\theta$ and $Q$ with respect to some common measure, e.g., $\frac{1}{2} P_\theta+ \frac{1}{2} Q$. The inner supremum is attained by $D^*(x) = \frac{p(x)}{p(x)+q(x)}$. Plugging in the formula for $D^*(x)$, we obtain
%$$ \E_{P_\theta} \log \frac{p(X)}{p(X)+q(X)} + \E_{Q} \log \frac{q(X)}{p(X)+q(X)}, $$
the new formulation
\begin{equation*}
\inf_G \left\{2\text{JS}(P_\theta\|Q)-\log 4\right\},
\end{equation*}
where $\text{JS}(P_\theta\|Q)$ is the Jensen-Shannon divergence between $P_\theta$ and $Q$. 
%Thus, the GAN model is equivalent to minimizing $\inf_G \text{JS}(P_\theta\|Q)$.
Thus, we can view the learned distribution $Q$ as an approximation to the true distribution, from which an estimator is subsequently derived. For example, for learning the mean of a true distribution, we can use the mean of $Q$ as our estimator.

After the original formulation of the GAN model, further generalizations have appeared in the literature by changing the divergence. The TV-GAN is defined by $\inf_{\theta} TV(P_{\thetastar}, P_{\theta})$, where $P_{\thetastar}$ is the true distribution and $P_{\theta}$ is the estimated distribution. By \cite{Nguyen1}, we have $TV(P, Q)=\sup_{f(\cdot)\in[0,1]} \left\{\E_P f(X) - \E_{Q} f(X)\right\}$, where $f(\cdot)$ is a function bounded in $[0,1]$. 

When training the TV-GAN model, the unknown $P_\theta$ is replaced by the empirical distribution $P_n$ of the observed data. Thus, the TV-GAN-based estimator is defined by
\begin{equation}
\widehat{\theta} = \argmin_{\theta} \sup_{f(\cdot)\in[0,1]} \left\{\E_{P_n}f(X) - \E_{P_\theta} f(X)\right\}. \label{tv3}
\end{equation}
%Besides, the function class is implemented by neural network functions.
%Under Huber's contamination model, Gao et al.~\cite{Gao1} gives upper bound of TV-GAN estimator's convergence rate, which is same as Huber's median's rate. %More precisely, not exactly TV-GAN, since they consider function class D as logistic regression, not all $D(\cdot)\in[0,1].$ \\
Similarly, a W-GAN is defined by $\inf_{\theta} W(P_{\thetastar}, P_{\theta})$, where $P_{\thetastar}$ is the true distribution and $P_\theta$ is the estimated distribution. 
%By Kantorovich duality \citep{villani_cedric_2009}, 
%\begin{equation*}
%W_p(P, Q)=\sup_{f(x)-g(y)\leq \|x-y\|^p} \E_P f(X) - \E_{Q} g(X).
%\end{equation*}
%For the special case of $p=1$, we have the following simple representation \citep{villani_cedric_2009}:
%\begin{equation}
%W_1(P, Q)=\sup_{\|f(x)\|_L\leq 1} \E_P f(X) - \E_{Q} f(X). \label{Wgan}
%\end{equation}
%In this paper, we focus our attention on the Wasserstein distance of order 1, i.e., $W_1(P,Q)$, so we omit the subscript and use $W(P,Q)$ to represent the Wasserstein-1 distance throughout the paper.

When training a W-GAN model, we again need to use observed data. Thus, using the duality form for the Wasserstein-1 distance~\eqref{Wgan}, the W-GAN-based estimator is given by
\begin{equation}
\widehat{\theta} = \argmin_{\theta} \sup_{\|f(x)\|_L\leq B} \left\{\E_{P_n}f(X) - \E_{P_\theta}f(X)\right\}, \label{w2}
\end{equation}
where $P_n$ is the empirical distribution of the observed data. 

%Note that when we change Lipschitz-1 continuous functions in~\eqref{Wgan} to Lipschitz-B continuous functions, where $B>0$ is come constant, \eqref{Wgan} just becomes a scaled version of Wasserstein distance and the estimator~\eqref{w2} is unchanged. However, it will allow for easier implementation using neural networks, as we will see later. Then the estimator just becomes
%\begin{equation}
%\widehat{\theta} = \argmin_{\theta} \sup_{\|f(x)\|_L\leq B} \left\{\E_{P_n}f(X) - \E_{P_\theta}f(X)\right\}. \label{w3}
%\end{equation}

%%%%%

\section{General upper and lower bounds}

In this appendix, we provide proofs of the general theorems involving upper and lower bounds, as well as proofs of supporting technical lemmas.

%%%%%

\subsection{Properties of neural network function classes}
\label{AppNN}

We begin by proving several properties of the function classes used in our neural network constructions.

\begin{lemma}
\label{LemNNLip}
For any $f^{(L)} \in \mathcal{D}(B,L)$ and $x, y \in \mathbb{R}^p$, we have $|f^{(L)}(x) - f^{(L)}(y)| \leq \frac{B}{4}\|x - y\|_2$ and $|f^{(L)}(x) - f^{(L)}(y)| \leq 2$.
\end{lemma}

\begin{proof}
We first prove the result for nodes in layers 1 through $L$ by induction. Note that for the sigmoid function $\sigma(x)=\frac{1}{1+e^{-x}}$, we have $|\sigma'(x)| \leq \frac{1}{4}$. Thus, for the first layer ($h=1$) and for any node $j$, we have
\begin{equation*}
|f_j^{(1)}(x) - f_j^{(1)}(y)| = |\sigma(w^T x + b) - \sigma(w^T y + b)| \leq \frac{1}{4}|w^T(x-y)|\leq \frac{1}{4}\|w\|_2\|x-y\|_2 \leq \frac{B}{4}\|x-y\|_2.
\end{equation*}
Boundedness is obvious: $|f^{(1)}(x) - f^{(1)}(y)| \leq 2$.

Now assume the bounds hold for all nodes in layer $h$, where $1 \leq h < L-1$. For node $j$ in layer $h+1$, and for any $x,y \in \mathbb{R}^p$, we have
\begin{align}
\left|f_j^{(h+1)}(x) - f_j^{(h+1)}(y)\right| & = \left|\text{ReLU}\left(\sum\limits_{i=1}^{d^{(h)}} w_i f^{(h)}_i(x)\right) - \text{ReLU}\left(\sum\limits_{i=1}^{d^{(h)}} w_i f^{(h)}_i(y)\right)\right|  \notag  \\
&\leq \left|\sum\limits_{i=1}^{d^{(h)}} w_i f^{(h)}_i(x) - \sum\limits_{i=1}^{d^{(h)}} w_i f^{(h)}_i(y)\right| \notag \\
& \leq \|w\|_1 \|f^{(h)}(x) - f^{(h)}(y)\|_\infty \notag \\ 
& \leq \max_i \left|f_i^{(h)}(x) - f_i^{(h)}(y)\right| \label{lem1.4}   \\
& \leq \frac{B}{4}\|x-y\|_2 \notag,
\end{align}  
where the first inequality holds because $|\text{ReLU}(u) - \text{ReLU}(v)| \leq |u-v|$, the second inequality uses Holder's inequality, the third inequality holds because $\|w\|_1 \leq 1$ for layers higher than 1, and the last inequality holds by the inductive hypothesis. From the expression in inequality~\eqref{lem1.4} and the inductive hypothesis, we also have $|f_j^{(h+1)}(x) - f_j^{(h+1)}(y)| \leq 2$.

For the last layer, note that
\begin{equation*}
\left|f_j^{(L)}(x) - f_j^{(L)}(y)\right| =  \left|\sum\limits_{i=1}^{d^{(L-1)}} w_i f^{(L-1)}_i(x) - \sum\limits_{i=1}^{d^{(L-1)}} w_i f^{(L-1)}_i(y)\right|,
\end{equation*}
so the same argument used above and the result for $h = L-1$ gives the desired conclusion.
\end{proof}

The same argument used in the proof of Lemma~\ref{LemNNLip} leads to the following corollary:

\begin{corollary} 
\label{CorNNLip}
For any $f^{(L)} \in \mathcal{D}_{sp}(B, L, k)$ and any $x, y \in \mathbb{R}^p$, we have $|f^{(L)}(x) - f^{(L)}(y)| \leq \frac{B}{4}\|x - y\|_2$ and $|f^{(L)}(x) - f^{(L)}(y)| \leq 2$.
\end{corollary}

Next, we provide uniform concentration bounds for the two function classes, which play a key role in deriving the main theorems for the upper bounds. The proofs are contained in Appendices~\ref{AppLemConcentration} and~\ref{AppCorConcentration}.

%\textbf{Lemma 2.} \textit{Assume $X_1, ..., X_n$ are i.i.d.\ random vectors, and the function class $\mathcal{D}$ defined by (\eqref{func1}) and (\eqref{func2}), then for any $\delta > 0$, we have } 
%\begin{equation*}
%\sup_{D \in \mathcal{D}}|\frac{1}{n}\sum_{i=1}^{n} D(X_i) - \mathbb{\E}D(X)| \leq C(\sqrt{\frac{q\log q}{n}} + \sqrt{\frac{2\log(1/\delta)}{n}})
%\end{equation*}
%\textit{with probability at least $1 - \delta$}.

\begin{lemma}
\label{LemConcentration}
Assume $X_1, \dots, X_n \in \mathbb{R}^p$ are i.i.d.\ random vectors following any distribution $P$. For any $\delta > 0$, we have
\begin{equation*}
\sup_{D \in \mathcal{D}(B,L)} \left|\frac{1}{n}\sum_{i=1}^{n} D(X_i) - \mathbb{\E}D(X)\right| \leq C\left(2^{L}\sqrt{\frac{p}{n}} + \sqrt{\frac{2\log(1/\delta)}{n}}\right),
\end{equation*}
with probability at least $1 - \delta$, where $C > 0$ is a universal constant.
\end{lemma}

\begin{corollary}
\label{CorConcentration}
Assume $X_1, \dots, X_n \in \mathbb{R}^p$ are i.i.d.\ random vectors following any distribution $P$. For any $\delta > 0$, we have
\begin{equation*}
\sup_{D \in \mathcal{D}_{sp}(B, L, k)} \left|\frac{1}{n}\sum_{i=1}^{n} D(X_i) - \mathbb{\E}D(X)\right| \leq C\left(2^{L}\sqrt{\frac{k\log\frac{p}{k}}{n}} + \sqrt{\frac{2\log(1/\delta)}{n}}\right),
\end{equation*}
with probability at least $1 - \delta$, where $C > 0$ is a universal constant.
\end{corollary}

Note that the bounds in Lemma~\ref{LemConcentration} and Corollary~\ref{CorConcentration} do not depend on $B$, since we only need the boundedness of the function classes, rather than the Lipschitz property, to complete the proofs.

%%%%%

\subsection{Proof of Lemma~\ref{LemConcentration}}
\label{AppLemConcentration}
For simplicity, we use $\mathcal{D}$ to denote $\mathcal{D}(B, L)$. Let 
$$f(x_1, \dots, x_n) = \sup_{D \in \mathcal{D}} \left|\frac{1}{n}\sum_{i=1}^{n} D(x_i) - \mathbb{E}D(X)\right|.$$ 
For a given $(x_1, \dots, x_j \dots, x_n)$, assume the supremum is attained at $D^*$. Then
\begin{equation*}
f(x_1, \dots, x_j \dots, x_n) =  \left|\frac{1}{n}\sum_{i=1}^{n} D^*(x_i) - \mathbb{E}D^*(X)\right|,
\end{equation*}
and for any  $(x_1, \dots, x_j' \dots, x_n)$, we have
\begin{align*}
& f(x_1, \dots, x_j \dots, x_n) - f(x_1, \dots, x_j' \dots, x_n)  \\
& \qquad =  \left|\frac{1}{n}\sum_{i=1}^{n} D^*(x_i) - \mathbb{E}D^*(X)\right| -  \sup_{D \in \mathcal{D}} \left|\frac{1}{n}\sum_{i=1, i\neq j}^{n} D(x_i) + \frac{1}{n}D(x_j') - \mathbb{E}D(X)\right| \\
& \qquad \leq \left|\frac{1}{n}\sum_{i=1}^{n} D^*(x_i) - \mathbb{E}D^*(X)\right| -  \left|\frac{1}{n}\sum_{i=1, i\neq j}^{n} D^*(x_i) + \frac{1}{n}D^*(x_j') - \mathbb{E}D^*(X)\right| \\
& \qquad \leq \left|\left(\frac{1}{n}\sum_{i=1}^{n} D^*(x_i) - \mathbb{E}D^*(X)\right) - \left(\frac{1}{n}\sum_{i=1, i\neq j}^{n} D^*(x_i) + \frac{1}{n}D^*(x_j') - \mathbb{E}D^*(X)\right)\right| \\
& \qquad = \left|\frac{1}{n}D^*(x_j) - \frac{1}{n}D^*(x_j')\right|\\
& \qquad \leq   \frac{2}{n},
\end{align*}
where the second inequality comes from the triangle inequality and the last inequality follows from Lemma~\ref{LemNNLip}. Similarly, we have
\begin{equation*}
f(x_1, \dots, x_j' \dots, x_n) - f(x_1, \dots, x_j \dots, x_n) \leq \frac{2}{n}.
\end{equation*}
Thus, by McDiarmid's inequality, we have 
\begin{equation}
f(X_1, \dots, X_n) \leq \mathbb{E} f(X_1, \dots, X_n) + \sqrt{\frac{2\log(1/\delta)}{n}}, \label{eqndiarmid}
\end{equation}
with probability at least $1-\delta$.

Furthermore, by the symmetrization technique~\citep{Ver18}, we have the bound
\begin{equation}
\mathbb{E}f(X_1, \dots, X_n) \leq 2 \mathbb{E} \sup_{D \in \mathcal{D}} \left| \frac{1}{n} \sum_{i=1}^{n} \epsilon_i D(X_i)\right|, \label{eqnsymm}
\end{equation}
where the $\epsilon_i$'s are independent Rademacher random variables. Note that since the nodes in the last layer do not involve an activation function, the function class $\mathcal{D}$ is symmetric. Thus,
\begin{equation}
\mathbb{E} \sup_{D \in \mathcal{D}} \left| \frac{1}{n} \sum_{i=1}^{n} \epsilon_i D(X_i)\right| = \mathbb{E} \sup_{D \in \mathcal{D}} \frac{1}{n} \sum_{i=1}^{n} \epsilon_i D(X_i). \label{eqnnoabs}
\end{equation}

For $1 < l < L-1$, consider the function class $\mathcal{D}^{(l+1)}$, defined in~\eqref{func2}, and let $b^{(l)}$ denote the width of layer $l$. Then we have 
\begin{align*}
\mathbb{E} \sup_{D \in \mathcal{D}^{(l+1)}} \frac{1}{n} \sum_{i=1}^{n} \epsilon_i D(X_i)
& = \mathbb{E} \sup_{\{D_j\} \subseteq \mathcal{D}^{(l)}, \|w\|_1 \le 1} \frac{1}{n} \sum_{i=1}^{n} \epsilon_i \text{ReLU}\left(\sum_{j=1}^{b^{(l)}} w_j D_j(X_i)\right) \\
& \leq \mathbb{E} \sup_{\{D_j\} \subseteq \mathcal{D}^{(l)}, \|w\|_1 \le 1} \frac{1}{n} \sum_{i=1}^{n} \epsilon_i \left(\sum_{j=1}^{b^{(l)}} w_jD_j(X_i) \right) \\
& = \mathbb{E} \sup_{\{D_j\} \subseteq \mathcal{D}^{(l)}, \|w\|_1 \le 1} \sum_{j=1}^{b^{(l)}} w_j \left( \frac{1}{n} \sum_{i=1}^{n} \epsilon_i D_j(X_i)\right)  \\
& \leq \mathbb{E} \sup_{\{D_j\} \subseteq \mathcal{D}^{(l)}} \max_{1 \le j \le b^{(l)}} \left|\frac{1}{n} \sum_{i=1}^{n} \epsilon_i D_j(X_i)\right|  \\
& = \mathbb{E} \sup_{D \in \mathcal{D}^{(l)}} \left|\frac{1}{n} \sum_{i=1}^{n} \epsilon_i D(X_i)\right|  \\
& \leq 2  \mathbb{E} \sup_{D \in \mathcal{D}^{(l)}} \frac{1}{n} \sum_{i=1}^{n} \epsilon_i D(X_i),
\end{align*}
where the first inequality comes from Talagrand’s contraction lemma \citep{Wai19}, since the ReLU function is Lipschitz-1 continuous; and the second inequality is from Holder's inequality with $p=1$ and $q=\infty$. The third inequality can be argued as follows:
let $A = \sup_{D \in \mathcal{D}^{(l)}} \left|\frac{1}{n} \sum_{i=1}^{n} \epsilon_i D(X_i)\right| = \sup_{D \in \mathcal{D}^{(l)}} \left|\frac{1}{n} \sum_{i=1}^{n} -\epsilon_i D(X_i)\right| $, and
note that $\sup_{D \in \mathcal{D}^{(l)}} \frac{1}{n} \sum_{i=1}^{n} \epsilon_i D(X_i) \geq 0$ and $\sup_{D \in \mathcal{D}^{(l)}} \frac{1}{n} \sum_{i=1}^{n} -\epsilon_i D(X_i) \geq 0$. 
Thus,
\begin{equation*}
\sup_{D \in \mathcal{D}^{(l)}} \frac{1}{n} \sum_{i=1}^{n} \epsilon_i D(X_i) + \sup_{D \in \mathcal{D}^{(l)}} \frac{1}{n} \sum_{i=1}^{n} -\epsilon_i D(X_i) \geq A.
\end{equation*}
Taking an expectation with respect to $\epsilon_i$ and noting that
\begin{equation*}
\E \left[\sup_{D \in \mathcal{D}^{(l)}} \frac{1}{n} \sum_{i=1}^{n} \epsilon_i D(X_i)\right] = \E\left[\sup_{D \in \mathcal{D}^{(l)}} \frac{1}{n} \sum_{i=1}^{n} -\epsilon_i D(X_i)\right]
\end{equation*}
by symmetry of the Rademacher distribution, we obtain the desired inequality.
Note that from the proof above, it is obvious that this string of inequalities also holds for the last layer (without ReLU functions).
Iterating this argument, we have
\begin{equation}
\label{EqnOrange}
\mathbb{E} \sup_{D \in \mathcal{D}} \frac{1}{n} \sum_{i=1}^{n} \epsilon_i D(X_i) \leq 2^{L-1}\mathbb{E} \sup_{D \in \mathcal{D}^{(1)}} \frac{1}{n} \sum_{i=1}^{n} \epsilon_i D(X_i),
\end{equation}
where $\mathcal{D}^{(1)}$ is the first-layer function class defined in equation~\eqref{func1}.

The final Rademacher complexity can be bounded by Dudley's integral entropy bound, which gives
\begin{equation}
\mathbb{E} \sup_{D \in \mathcal{D}^{(1)}} \frac{1}{n} \sum_{i=1}^{n} \epsilon_i D(X_i) \lesssim \mathbb{E} \frac{1}{\sqrt{n}} \int_0^2 \sqrt{\log \mathcal{N}(\delta, \mathcal{D}^{(1)}, \|\cdot\|_n)}d\delta,  \label{lemma2-p1}
\end{equation}
where $ \mathcal{N}(\delta, \mathcal{D}^1, \|\cdot\|_n)$ is the $\delta$-covering number of $\mathcal{D}^{(1)}$ with respect to the empirical $\ell_2$-distance. Since the VC dimension of $\mathcal{D}^{(1)}$ is $O(p)$ (cf.\ Example 4.21 in~\cite{Wai19}), we have $ \mathcal{N}(\delta, \mathcal{D}^{(1)}, \|\cdot\|_n) \lesssim p(16e/\delta)^{O(p)}$~\citep{VanWel96, Wai19}. This leads to the bound
\begin{equation}
\label{EqnBanana}
\frac{1}{\sqrt{n}} \int_0^2 \sqrt{\log \mathcal{N}(\delta, \mathcal{D}^{(1)}, \|\cdot\|_n)}d\delta \lesssim \sqrt{\frac{p}{n}}.
\end{equation}

Combining inequalities~\eqref{eqndiarmid}, \eqref{eqnsymm}, \eqref{eqnnoabs}, \eqref{EqnOrange}, \eqref{lemma2-p1}, and~\eqref{EqnBanana}, we arrive at the desired conclusion.
%$$ f(X_1, \dots, X_n) \leq 2^{L}\sqrt{\frac{p}{n}} + \sqrt{\frac{2\log(1/\delta)}{n}}, $$
%with probability at least $1-\delta$.

%%%%%

%%%%%

\subsection{Proof of Corollary~\ref{CorConcentration}}
\label{AppCorConcentration}

Note that with the new function class, by Corollary~\ref{CorNNLip}, the proof in Lemma~\ref{LemConcentration} still holds until (including) inequality~\eqref{lemma2-p1}. Then we need a new bound for the $\delta$-covering number for function class \eqref{func4}. 
Recall that the VC dimension of a real-valued function class is defined to be the VC dimension of the associated set class, as follows: For the function class $\mathcal{G} = \{g: \mathcal{X} \rightarrow \mathbb{R}\}$, the subgraph at level 0 for the function $g$ is defined as $S_g := \{x \in \mathcal{X} | g(x) \leq 0\}$, and the set class associated with $\mathcal{G}$ is defined as $\mathcal{S}(\mathcal{G}) := \{S_g, g \in \mathcal{G}\}$ (cf.\ Section 4.3.3 in \cite{Wai19}).

We first derive the following result:
\begin{lemma}
\label{VCdim_lemma}
The VC dimensions of the function classes~\eqref{func4} and~\eqref{func6} are bounded by $O(k\log\frac{p}{k})$ and $O(\max(k, \log p + k\log\frac{\log p}{k}))$, respectively.
\end{lemma}

\begin{proof}
Consider a finite union of function classes, $\mathcal{H} = \cup_{i=1,\dots, r}\mathcal{H}_i$, where the VC dimension of each $\mathcal{H}_i$ is $d$. By Sauer's theorem (cf.\ Theorem 4.18 of \cite{Wai19}), the growth function $\Gamma_{\mathcal{H}_i}(m)$ is bounded by 
$$ \Gamma_{\mathcal{H}_i}(m) \leq \sum_{i=0}^{d} {m\choose i}.$$
Thus, the growth function of $\mathcal{H}$ is bounded by 
$$ \Gamma_{\mathcal{H}}(m) \leq r\sum_{i=0}^{d} {m\choose i}.$$
Denote the VC dimension of $\mathcal{H}$ by $D$. By definition, $\Gamma_\mathcal{H}(D)=2^D$, so we have 
$$ 2^D \leq r\sum_{i=0}^{d} {D\choose i} \leq r \left(\frac{eD}{d}\right)^d.$$
Thus, $D \leq O\left(\max\left\{d, \log r + d\log\frac{\log r}{d}\right\}\right)$.

For the first-layer function classes~\eqref{func4} and~\eqref{func6}, each function class in the union has VC dimension $O(k-1)$, and $r$ is $\binom{p}{k} = O\left(\left(\frac{ep}{k}\right)^k\right)$ and $O(p)$, respectively. Thus, we have the VC bounds $O(k+k\log\frac{p}{k})=O(k\log\frac{p}{k})$ and $O\left(\max\{k, \log p + k\log\frac{\log p}{k}\}\right)$.
\end{proof}

By Lemma~\ref{VCdim_lemma}, the VC dimension for the function class \eqref{func4} is $O(k\log\frac{p}{k})$. Thus, following the same lines of proof for Lemma~\ref{LemConcentration}, we obtain the desired result.

%%%%%

\subsection{Proof of Theorem~\ref{Upper1}}
\label{AppUpper1}

We first bound $\sup_{f \in \mathcal{D}} \E_{P_{\thetastar}}f(X) - \E_{P_{\widehat{\theta}}}f(X)$:
\begin{align*}
\sup_{f \in \mathcal{D}} \E_{P_{\thetastar}}f(X) - \E_{P_{\widehat{\theta}}}f(X) & = \sup_{f \in \mathcal{D}} \left(\E_{P_{\thetastar}}f(X) - \E_{P}f(X)\right) + \left(\E_{P}f(X)- \E_{P_{\widehat{\theta}}}f(X)\right) \\
& \leq  \sup_{f \in \mathcal{D}} \left(\E_{P_{\thetastar}}f(X) - \E_{P}f(X)\right) +  \sup_{f \in \mathcal{D}} \left(\E_{P}f(X)- \E_{P_{\widehat{\theta}}}f(X)\right).
\end{align*}
By Lemma~\ref{LemNNLip} and the duality form of the Wasserstein distance~\eqref{Wgan}, this is further bounded by
\begin{align}
& \frac{B}{4} W(P_{\thetastar}, P) + \sup_{f \in \mathcal{D}} \left(\E_{P}f(X)- \E_{P_{\widehat{\theta}}}f(X)\right) \notag \\
& \qquad \leq \sup_{f \in \mathcal{D}} \left(\E_{P}f(X)- \E_{P_{\widehat{\theta}}}f(X)\right)  + \frac{B}{4}\epsilon \notag \\
& \qquad \leq \sup_{f \in \mathcal{D}} \left(\E_{P}f(X) - \E_{P_n}f(X)\right) +  \sup_{f \in \mathcal{D}} \left(\E_{P_n}f(X)- \E_{P_{\widehat{\theta}}}f(X)\right) + \frac{B}{4}\epsilon \notag \\
& \qquad \leq \sup_{f \in \mathcal{D}} \left(\E_{P_n}f(X)- \E_{P_{\widehat{\theta}}}f(X)\right)  + C\left(2^{L}\sqrt{\frac{p}{n}} + \sqrt{\frac{2\log(1/\delta)}{n}}\right) + \frac{B}{4}\epsilon,\label{lemma2-1}
\end{align}
where the last inequality holds by Lemma~\ref{LemConcentration}, with probability at least $1-\delta$.

By the optimality of $\widehat{\theta}$, the right-hand side of inequality~\eqref{lemma2-1} is further bounded by
\begin{align}
& \sup_{f \in \mathcal{D}} \left(\E_{P_n}f(X)- \E_{P_{\thetastar}}f(X)\right)  + C\left(2^{L}\sqrt{\frac{p}{n}} + \sqrt{\frac{2\log(1/\delta)}{n}}\right) + \frac{B}{4}\epsilon \notag \\
& \qquad \leq \sup_{f \in \mathcal{D}} \left(\E_{P_n}f(X)- \E_{P}f(X)\right)  + \sup_{f \in \mathcal{D}} \left(\E_{P}f(X)- \E_{P_{\thetastar}}f(X)\right) \notag \\
& \qquad \qquad + C\left(2^{L}\sqrt{\frac{p}{n}} + \sqrt{\frac{2\log(1/\delta)}{n}}\right) + \frac{B}{4}\epsilon \notag \\
& \qquad \leq 2 C\left(2^{L}\sqrt{\frac{p}{n}} + \sqrt{\frac{2\log(1/\delta)}{n}}\right)+ \frac{B}{2}\epsilon, \label{lemma2-2}
\end{align}
again applying the bound in Lemma~\ref{LemConcentration} and the duality form of the Wasserstein distance.

Note that for any $f \in \mathcal{D}$, we also have $-f \in \mathcal{D}$, so we similarly obtain
$$\sup_{f \in \mathcal{D}} \left(- \E_{P_{\thetastar}}f(X) + \E_{P_{\widehat{\theta}}}f(X)\right) \leq 2C\left(2^{L}\sqrt{\frac{p}{n}} + \sqrt{\frac{2\log(1/\delta)}{n}}\right)+ \frac{B}{2} \epsilon,$$
with probability at least $1-\delta$. Thus,
$$\sup_{f \in \mathcal{D}} \left|\E_{P_{\thetastar}}f(X) - \E_{P_{\widehat{\theta}}}f(X)\right| \leq 2C\left(2^{L}\sqrt{\frac{p}{n}} + \sqrt{\frac{2\log(1/\delta)}{n}}\right)+ \frac{B}{2} \epsilon,$$
with probability at least $1-\delta$.

%%%%%

\subsection{Proof of Theorem~\ref{Upper2}}
\label{AppUpper2}

Compared to Theorem~\ref{Upper1}, note that the only difference in this setting is the sparsity in the first layer of the neural network. This does not affect any lines in the proof of Theorem~\ref{Upper1}. The only change is in inequalities~\eqref{lemma2-1} and~\eqref{lemma2-2}, where the proof should invoke the uniform concentration inequality of Corollary~\ref{CorConcentration} rather than Lemma~\ref{LemConcentration}. This leads to the result of the theorem.
%We only need to replace it with to obtain the tighter bound for $\sup_{f \in \mathcal{D}} (\E_{P}f(X) - \E_{P_n}f(X))$ and $\sup_{f \in \mathcal{D}} (\E_{P_n}f(X) - \E_{P}f(X))$. Thus, we obtain the result of the theorem.

%%%%%

\subsection{Proof of Theorem~\ref{ThmModulus}}
\label{AppThmModulus}

Suppose $P_{\eta_1}$ and $P_{\eta_2}$ achieve the supremum loss among pairs of distributions separated by Wasserstein distance at most $\epsilon$. In other words, 
$$L(\eta_1, \eta_2) =  \sup_{W(P_{\theta_1}, P_{\theta_2})\leq \epsilon, \theta_1, \theta_2 \in \Theta} L({\theta_1}, {\theta_2}) = m(\epsilon, \Theta).$$ 
Let $P'=\frac{1}{2}P_{\eta_1} + \frac{1}{2}P_{\eta_2}$. By the duality form of the Wasserstein distance, we have
\begin{align*}
W(P', P_{\eta_1}) & = \sup_{\|f\|_L \leq 1} \left(\E_{P'} f(X) - \E_{P_{\eta_1}} f(X)\right) \\
& = \frac{1}{2}\sup_{\|f\|_L \leq 1} \left(\E_{P_{\eta_2}} f(X) - \E_{P_{\eta_1}} f(X)\right) \\
& = \frac{1}{2} W(P_{\eta_2}, P_{\eta_1}) \leq \epsilon. 
\end{align*}
Similarly, $W(P', P_{\eta_2}) \leq \epsilon$.

Thus, we have
\begin{align*}
& \sup_{\thetastar \in \Theta, P: W(P_{\thetastar}, P) \le \epsilon} \mprob(L(\thetastar, \thetahat) \ge \delta) \\
& \qquad \qquad \ge \frac{1}{2} \sup_{P: W(P_{\eta_1}, P) \le \epsilon} \mprob_{X_i \sim P} (L(\eta_1, \thetahat) \ge \delta) + \frac{1}{2} \sup_{P: W(P_{\eta_2}, P) \le \epsilon} \mprob_{X_i \sim P} (L(\eta_2, \thetahat) \ge \delta) \\
& \qquad \qquad \ge \frac{1}{2} \mprob_{X_i \sim P'} (L(\eta_1, \thetahat) \ge \delta) + \frac{1}{2} \mprob_{X_i \sim P'} (L(\eta_2, \thetahat) \ge \delta).
%
%& \qquad = \frac{1}{2} \left(\frac{1}{2} \mprob_{X_i \sim P_{\eta_1}} (L(\eta_1, \thetahat) \ge \delta) + \frac{1}{2} \mprob_{X_i \sim P_{\eta_2}} (L(\eta_1, \thetahat) \ge \delta)\right) \\
%
%& \qquad \qquad + \frac{1}{2} \left(\frac{1}{2} \mprob_{X_i \sim P_{\eta_1}} (L(\eta_2, \thetahat) \ge \delta) + \frac{1}{2} \mprob_{X_i \sim P_{\eta_2}} (L(\eta_2, \thetahat) \ge \delta)\right).
\end{align*}
Note that if we take $\delta = \frac{1}{2}m(\epsilon, \Theta)$, then
\begin{align*}
\mprob(L(\eta_1, \thetahat) \ge \delta) + \mprob(L(\eta_2, \thetahat) \ge \delta) & \ge \mprob\left(\left\{L(\eta_1, \thetahat) \ge \delta\right\} \bigcup \left\{L(\eta_2, \thetahat) \ge \delta\right\}\right) \\
& \ge \mprob(L(\eta_1, \eta_2) \ge 2\delta) = 1,
\end{align*}
using the fact that $L$ satisfies the triangle inequality.
Thus, we conclude that
\begin{equation*}
\sup_{\thetastar \in \Theta, P: W(P_{\thetastar}, P) \le \epsilon} \mprob\left(L(\thetastar, \thetahat) \ge \frac{m(\eta, \Theta)}{2}\right) \ge \frac{1}{2}.
\end{equation*}
Since the preceding chain of inequalities holds for any $\thetahat$, we may also take an infimum over $\thetahat$ on the left-hand side.

Furthermore, we clearly have
\begin{equation*}
\inf_{\thetahat} \sup_{\thetastar \in \Theta, P: W(P_{\thetastar}, P) \le \epsilon} \mprob_{X_i \sim P}\left(L(\thetastar, \thetahat) \ge \mathcal{M}(0)\right) \ge \inf_{\thetahat} \sup_{\thetastar \in \Theta} \mprob_{X_i \sim P_{\thetastar}}\left(L(\thetastar, \thetahat) \ge \mathcal{M}(0)\right) \ge c.
\end{equation*}
Thus, we have
\begin{equation*}
\inf_{\thetahat} \sup_{\thetastar \in \Theta, P: W(P_{\thetastar}, P) \le \epsilon} \mprob_{X_i \sim P}\left(L(\thetastar, \thetahat) \ge \frac{m(\eta, \Theta)}{2} \wedge \mathcal{M}(0)\right) \ge \frac{1}{2} \vee c,
\end{equation*}
completing the proof.

%For the two distributions $P_{\eta_1}$ and $P_{\eta_2}$ which have loss of $m(\epsilon, \Theta)$, we can find a common valid perturbation, $P$, which means $P_{\eta_1}$ and $P_{\eta_2}$ are not identifiable under the model. Then we can follow standard Le Cam's two point testing procedure to get the result.

%Suppose $\{\theta_1, ..., \theta_M\}$ is a $2\delta$-separated set, meaning $L(\theta_i, \theta_j) \geq 2\delta$ for any $i \neq j$. Suppose $J$ is uniform random variable from $\{1,...,M\}$. Given $J=j$, sample $Z$ from distribution $P_j$, where $W(P_j, P_{\theta_j}) \leq \epsilon$. We have, for any $\delta$, 
%%%
%\begin{equation}
%\sup_{\thetastar \in \Theta, P:W(P_{\thetastar},P)\leq \epsilon} \mathbb{P}(L(\thetastar, \widehat{\theta}) \geq \delta) \geq \inf_\psi \mathbb{P}(\psi(Z)\neq J) \label{modu1}
%\end{equation}
%%%
%where $\psi$ is a testing function, $\psi: \mathcal{Z} \rightarrow [M]$ and right hand side probability is taken over joint distribution of $(J, Z)$. See Proposition 15.1 in \cite{Wai19} for proof.

%Le Cam's method takes a separated set with two points. And we have (equation (15.13) in  \cite{Wai19})
%%%
%\begin{equation}
%\inf_\psi \mathbb{P}(\psi(Z)\neq J) = \frac{1}{2}(1-\|P_{1} - P_{2}\|_{TV}) \label{modu2}
%\end{equation}
%%%
%Note when we pick $\eta_1$ and $\eta_2$ as the separated set, with $\delta = \frac{1}{2}m(\epsilon, \Theta)$, and let $P_1 = P_2 = P$, then we have
%%%
%\begin{equation}
%\frac{1}{2}(1-\|P_{1} - P_{2}\|_{TV}) = \frac{1}{2} \label{modu3}
%\end{equation}
%Combining \eqref{modu1}, \eqref{modu2} and \eqref{modu3}, we get the result.

%%%%%

\section{Proofs for location estimation}

In this appendix, we derive the upper and lower bounds specific to location estimation.

\subsection{Proof of Theorem~\ref{LocUpper}}
\label{AppLocUpper}

We pick a specific setup of weight parameters for the neural network function. For the first layer, let $u \in \real^p$ be such that $\|u\|_2 = B$, and define
\begin{equation*}
w^{(1)}_i = 
\begin{cases}
u, & \text{if } i = 1, \\
0, & \text{otherwise},
\end{cases}
\end{equation*}
and
\begin{equation*}
b^{(1)}_i = 
\begin{cases}
-u^T \thetastar, & \text{if } i = 1, \\
0, & \text{otherwise}.
\end{cases}
\end{equation*}
For $l > 1$, define
\begin{equation*}
w^{(l)}_i =
\begin{cases}
(1,0, \dots,0), & \text{if } i = 1, \\
0, & \text{otherwise}.
\end{cases}
\end{equation*}
Then the neural network output is $f(x) = \sigma(u^T(x-\thetastar))$. (Note that by our construction, all the intermediate outputs in the inner layers are positive, so the ReLU function is simply the identity.) Furthermore, for $X\sim \mathcal{N}(\theta, I)$, we have 
$$\E f(X) = \E\sigma(u^T(X-\thetastar)) = \E\sigma(u^T(X-\theta)+u^T(\theta-\thetastar)) =g(u^T(\theta-\thetastar)),$$ 
where we define $g(t) := \int_{-\infty}^{\infty} \sigma(Bz+t)\phi (z)dz$ and $\phi(\cdot)$ denotes the density function for $\mathcal{N}(0,1)$.

By Theorem~\ref{Upper1} with $\delta = e^{-(p+n\epsilon^2)}$, we have 
\begin{align*}
\sup_{u: \|u\|_2=B} |g(0) - g(u^T(\thetahat - \thetastar))| & \leq \sup_{f \in \mathcal{D}} |\E_{P_{\thetastar}}f(X) - \E_{P_{\widehat{\theta}}}f(X)| \\
& \leq 2C\left(2^{L}\sqrt{\frac{p}{n}} + \sqrt{\frac{2\log(1/\delta)}{n}}\right)+ \frac{B\epsilon}{2} \\
& = 2C\left(2^{L}\sqrt{\frac{p}{n}} + \sqrt{2\left(\frac{p}{n}+\epsilon^2\right)}\right)+ \frac{B\epsilon}{2}. %\label{LocUpper3}
\end{align*}
%Note for $\epsilon < \min(\frac{1}{2B}, \frac{1}{4\sqrt{2}C})$, there exist $N$ such that for any $n > N$, $\eqref{LocUpper3} < \frac{1}{2}$. 
Note that $g'(t) =\int_{-\infty}^{\infty} \sigma'(Bz+t)\phi (z)dz > 0$, for all $t \in \mathbb{R}$. %and $g'(0)>0$.

Thus, there exist constants $c_0, c_1 > 0$ such that $c_1|t| \leq |g(0)-g(t)|$ whenever $|g(0)-g(t)| \le c_0$ (where both $c_0$ and $c_1$ depend on $B$). Furthermore,
\begin{equation*}
2C\left(2^{L}\sqrt{\frac{p}{n}} + \sqrt{2\left(\frac{p}{n}+\epsilon^2\right)}\right)+ \frac{B\epsilon}{2} \le c_0
\end{equation*}
under the sample size requirement $2^L \cdot \frac{p}{n} + \epsilon^2 \le c$.
Thus, we have
\begin{align}
\label{EqnLocParamUpper}
\|\widehat{\theta} - \thetastar \|_2 & = \frac{1}{B} \sup_{u: \|u\|_2=B} |u^T(\widehat{\theta}-\thetastar)| \notag \\
& \leq \frac{1}{Bc_1} \sup_{u: \|u\|_2=B} |g(0) - g(u^T(\thetahat - \thetastar))| \notag \\
& \leq \frac{2C}{Bc_1}\left(2^{L}\sqrt{\frac{p}{n}} + \sqrt{2\left(\frac{p}{n}+\epsilon^2\right)}\right)+ \frac{\epsilon}{2c_1} \notag \\
& \le C' \left(2^L \sqrt{\frac{p}{n}} \vee \epsilon\right),
\end{align}
with probability at least $1-e^{-(p+n\epsilon^2)}$, where $C'$ is a constant only depending on $B$.

\subsection{Proof of Theorem~\ref{LocLower}}
\label{AppLocLower}

%By Theorem~\ref{ThmModulus}, we need to calculate $m(\epsilon, \Theta)$ and $\mathcal{M}(0)$. 

We first show that $m(\epsilon, \Theta) \geq \epsilon$ for the loss function $L(P_{\theta_1}, P_{\theta_2}) = \|\theta_1 - \theta_2\|_{2}$. Consider two distributions $\mathcal{N}(\eta_1, I)$ and $\mathcal{N}(\eta_2, I)$ such that  $\|\eta_1 - \eta_2\|_2 = \epsilon$, and consider the coupling $(Z,Z+\eta_2- \eta_1)$, where $Z\sim \mathcal{N}(\eta_1, I)$. Then by the definition of the Wasserstein-1 distance, we have $W_1(P_{\eta_1}, P_{\eta_2})  \leq \|\eta_1 - \eta_2\|_{2}.$
Thus,
\begin{align*} 
m(\epsilon, \Theta) & =  \sup_{W(P_{\theta_1}, P_{\theta_2})\leq \epsilon, \theta_1, \theta_2 \in \Theta}L(P_{\theta_1}, P_{\theta_2}) \geq \|\eta_1 - \eta_2\|_{2} = \epsilon.
\end{align*}
%For location estimation, under squared error loss, $\mathcal{M}(0) = \frac{p}{n}$, thus we get a lower bound as $\frac{p}{n} \vee \epsilon^2$.

Next, we calculate $\mathcal{M}$(0) using a fairly standard argument via Fano's method. From (the proof of) Propositions 15.1 and 15.12 in~\cite{Wai19}, we have the following guarantee for any estimator $\thetahat$, any $\delta>0$, and any $2\delta$-separated set $\mathcal{V} = \{\theta^1, \dots, \theta^M\} \subseteq \Theta$:
%%%
\begin{equation}
\label{EqnFano}\sup_{\thetastar \in \Theta} \mathbb{P}(\|\thetahat - \thetastar\|_2 \geq \delta) \geq 1 - \frac{I(Z;J) + \log2}{\log M},
\end{equation}
%\inf_{\psi}\mathbb{P}(\psi(X)\neq V),$$
%%%
where $J$ is chosen uniformly at random from $\{1, \dots, M\}$ and $Z \mid J = j \sim \mprob_{\theta^j}$. Furthermore, the mutual information may be bounded via
$$ I(Z;J) \leq \frac{1}{M^2}\sum_{j,k=1}^M \text{KL}(P_{\theta^j},P_{\theta^k}).$$
%the left probability is taken over the observed data $X=(X_1, \dots, X_n)$; the right probability is taken over the joint distribution of $(X,V)$; $V$ is uniform from $\mathcal{V}$; and $\Psi$ is a test. Note that $\mathcal{V}$ is a subset of $\Theta$.
%
%Fano's method is based on the following inequality:
%$$ \inf_{\psi}\mathbb{P}(\psi(X)\neq V) \geq 1 - \frac{I(V;X) + \log2}{\log|\mathcal{V}|}.$$
%For a carefully chosen $\delta$ and the packing set $\mathcal{V}$, we can bound the right-hand side by a constant. This is the idea of the local version of Fano's method. The $2\delta$-local packing set is a $2\delta$-separated set, but with another additional property: the KL-divergence between any two distributions in the set is uniformly bounded by $\kappa^2\delta^2$. Since the mutual information can be bounded by the KL-divergence, with this local packing set, we can hope to lower-bound the right-hand side.

Let $\mathcal{V}_0$ be a $1/2$-packing of the unit ball with cardinality $2^p$ (which exists by Example 5.8 in \cite{Wai19}). We then scale the vectors by $\delta$ to obtain $\mathcal{V} = \delta\mathcal{V}_0$. Note that for any $\theta^j, \theta^k \in \mathcal{V}$, we have $\|\theta^j - \theta^k\|_2 \le 2\delta$, so the formula for the KL divergence between multivariate Gaussian distributions gives $\text{KL}(P_{\theta_i},P_{\theta_j}) \le \frac{4n\delta^2}{2}$.
%Since Gaussian means are within a $\delta$-ball, the KL divergence of any two distributions of this set is bounded by $\frac{n\delta^2}{2}$, so
%and by the bound on mutual information (see equation (15.34) in \cite{Wai19}),
%%%
%$$ I(V;X) \leq \frac{1}{M^2}\sum_{j,k=1}^M \text{KL}(P_{\theta_i},P_{\theta_j}).$$
%%%
Thus, $ I(Z;J) \leq 2n\delta^2$, as well. Choosing $\delta^2 = \frac{p\log2}{2n}$, we conclude that
%$$1 - \frac{I(Z;J) + \log2}{\log M} \geq \frac{1}{4}.$$
%Then for any $\theta$, we have
$$\sup_{\thetastar \in \Theta} \mprob\left(\|\thetahat - \thetastar\|_2 \geq C\sqrt{\frac{p}{n}}\right) \geq 1 - \frac{p \log 2 + \log 2}{p} \ge \frac{1}{4},$$
where $C = \sqrt{\frac{\log2}{2}}$, so we can take $\mathcal{M}(0) = C \sqrt{\frac{p}{n}}$.

%%%%%   For constrained parameter space.   %%%%%%
%However, as we can see from the proof, all we need is this local packing set $\mathcal{V}$, in the true parameter space $\Theta$, the proof holds. And this $\mathcal{V}$ is a subset of ball of radius $\delta = \sqrt{\frac{p\sigma^2\log2}{2n}}$. Thus, even if we restrict our $\Theta$ to some other smaller space, e.g. $\{\theta | \|\theta\|_2 \leq B\}$, as long as this $B$ is larger than $\sqrt{\frac{p\sigma^2\log2}{2n}}$, we can choose the same $\mathcal{V}$ as before and the arguments hold. So the lower bound for the true parameter space $\Theta = \{\theta | \|\theta\|_2 \leq B\}$ is $O(\frac{p}{n})$.

Combining the expressions for $\mathcal{M}(0)$ and $m(\epsilon, \Theta)$ and using Theorem~\ref{ThmModulus} gives the desired result.

\subsection{Proof of Theorem~\ref{SparseLocUpper}}
\label{AppSparseLocUpper}

We define the same function $g(t)$ as in the proof of Theorem~\ref{LocUpper} and pick the same setup of weight parameters. Note that this setup of weight parameters satisfies the sparsity constraint, thus belongs to the function class $\mathcal{D}_{sp}(B, L, k)$. Then the argument used in the proof for Theorem~\ref{LocUpper} can still be applied, except we replace Theorem~\ref{Upper1} by Theorem~\ref{Upper2}, to obtain
\begin{equation*}
\|\widehat{\theta} - \theta \|_2 \leq C' \left(2^L \sqrt{\frac{k\log\frac{p}{k}}{n}} \vee \epsilon \right),
\end{equation*}
with probability at least $1-e^{-(p+n\epsilon^2)}$, where $C'$ only depends on $B$.

%%%%%

\subsection{Proof of Theorem~\ref{SparseLocLower}}
\label{AppSparseLocLower}

As in the proof of Theorem~\ref{LocLower}, we choose two distributions $\mathcal{N}(\eta_1, I)$ and $\mathcal{N}(\eta_2, I)$, where $\eta_1$ and $\eta_2$ are $k$-sparse vectors and $\|\eta_1 - \eta_2\|_2 = \epsilon$. Following the same arguments, we have $m(\epsilon, \Theta) \geq \epsilon$.
%and for sparse location estimation, we have $\mathcal{M}(0) = \frac{k\log\frac{p}{k}}{n}$ \citep{sparse_loc1}. Thus, we obtain the lower bound $\Omega\left(\sqrt{\frac{k\log\frac{p}{k}}{n}} \vee \epsilon\right)$.

To compute $\mathcal{M}$(0), we again use the bound via Fano's method~\eqref{EqnFano}.
%similar as before, we prove it via Fano's method. For any estimator $\widehat{\theta}$, any $\delta>0$, and any $2\delta$-packing (separated) set $\mathcal{V}$, we have
%$$ \sup_{\theta \in \Theta} \mprob\left(\|\widehat{\theta} - \theta\|_2 \geq \delta\right) \geq 1 - \frac{I(V;X) + \log2}{\log|\mathcal{V}|},$$
%where the left probability is taken over the observed data $X=(X_1, \dots, X_n)$ and $V$ is uniform on $\mathcal{V}$.
By Example 15.16 of \cite{Wai19}, we can find a 1/2-packing $\mathcal{V}_1$ of the unit ball of sparse vectors $S(k) = \{\theta \in \mathbb{R}^p: \|\theta\|_0 \leq k, \|\theta\|_2\leq 1\}$ of cardinality at least $\log|\mathcal{V}_1| \geq \frac{k}{2}\log\frac{p-k}{k}$. Then taking $\mathcal{V} = \delta \mathcal{V}_1$, we have $I(J;Z) \leq 2n\delta^2$, as before. 

Thus, if we take $\delta^2 = \frac{1}{8n}\frac{k}{2}\log\frac{p-k}{k}$, we have 
$$ \sup_{\theta \in \Theta} \mprob\left(\|\widehat{\theta} - \theta\|_2 \geq \delta\right) \geq 1 - \frac{1}{4} - \frac{\log2}{\frac{k}{2}\log\frac{p-k}{k}}.$$
As long as $k \geq 4$ and $k < \frac{p}{3}$, we then have
$$\sup_{\theta \in \Theta} \mprob\left(\|\widehat{\theta} - \theta\|_2 \geq \delta\right) \geq \frac{1}{4},$$
so $\mathcal{M}(0) = O\left(\frac{k\log\frac{p}{k}}{n}\right)$.

Theorem~\ref{ThmModulus} completes the proof.

%%%%%

\subsection{Proof of Theorem~\ref{EllipLoc}}
\label{AppEllipLoc}

Note that our proof of Theorem~\ref{Upper1} does not depend on the true distribution family. Thus, the same result holds for elliptical distributions, where the parameters $(\thetahat, \widehat{h}, \widehat{A})$ take the place of $\thetahat$.

Thus, with at least $1 - \delta$ probability, we have
$$\sup_{f \in \mathcal{D}} \left|\E_{\thetastar, h^*, A^*}f(X) - \E_{\widehat{\theta}, \widehat{h},\widehat{A}}f(X)\right| \leq 2 C\left(2^{L}\sqrt{\frac{p}{n}} + \sqrt{\frac{2\log(1/\delta)}{n}}\right)+  \frac{B}{2}\epsilon.$$

We pick a specific setup of weight parameters.
%(also appearing in \cite{Gao1}).
Let $\Sigma=AA^T$ and $\widehat{\Sigma}=\widehat{A}\widehat{A}^T.$
For the first layer, let $u \in \real^p$ be such that $\|u\|_2 = 1$, and define
\begin{equation*}
w^{(1)}_i = 
\begin{cases}
\frac{u}{\sqrt{u^T \Sigma^* u}}, & \text{if } i = 1, \\
0, & \text{otherwise},
\end{cases}
\end{equation*}
and
\begin{equation*}
b^{(1)}_i = 
\begin{cases}
\frac{-u^T \widehat{\theta}}{\sqrt{u^T \Sigma^* u}}, & \text{if } i = 1, \\
0, & \text{otherwise}.
\end{cases}
\end{equation*}

For $l > 1$, define
\begin{equation*}
w^{(l)}_i =
\begin{cases}
(1,0, \dots,0), & \text{if } i = 1, \\
0, & \text{otherwise}.
\end{cases}
\end{equation*}

Under this setup, the neural network output is $f(x) = \sigma\left(\frac{u^T(x-\widehat{\theta})}{\sqrt{u^T\Sigmastar u}}\right)$ (since the ReLU functions are applied only to positive inputs), and $\E f(X) = \E\sigma\left(\frac{u^T(X-\widehat{\theta})}{\sqrt{u^T\Sigmastar u}}\right).$ 

Then we have
\begin{align*}
\E_{\thetastar, h^*, A^*}f(X) - \E_{\widehat{\theta}, \widehat{h},\widehat{A}}f(X) & = \E_{\thetastar, h^*, A^*} \sigma\left(\frac{u^T(X-\thetahat)}{\sqrt{u^T\Sigma^* u}}\right) -  \E_{\widehat{\theta}, \widehat{h},\widehat{A}} \sigma\left(\frac{u^T(X-\thetahat)}{\sqrt{u^T\Sigma^* u}}\right) \\
& = \E_{\thetastar, h^*, A^*} \sigma\left(\frac{u^T\xi A^*U}{\sqrt{u^T \Sigma^* u}} + \frac{u^T (\thetastar - \thetahat)}{\sqrt{u^T \Sigma^* u}}\right) -  \E_{\widehat{\theta}, \widehat{h},\widehat{A}} \sigma\left(\frac{u^T\xi \widehat{A}U}{\sqrt{u^T \Sigma^* u}}\right) \\
& = \E\sigma\left(S + \frac{u^T(\thetastar - \thetahat)}{\sqrt{u^T \Sigma^* u}}\right) - \E\sigma(S),
\end{align*}
where we set $S = u^T \xi U$ and use the fact that
\begin{equation*}
\E_{\widehat{\theta}, \widehat{h},\widehat{A}} \sigma\left(\frac{u^T\xi \widehat{A}U}{\sqrt{u^T \Sigma^* u}}\right) = \frac{1}{2} = \E\sigma(S).
\end{equation*}

Thus,
\begin{align*}
|\E_{\thetastar, h^*, A^*}f(X) - \E_{\widehat{\theta}, \widehat{h},\widehat{A}}f(X)| & = \left|\E\sigma(S) - \E\sigma\left(S + \frac{u^T(\thetastar-\widehat{\theta})}{\sqrt{u^T\Sigmastar u}}\right)\right| \\
& = \left|g(0) - g\left(\frac{u^T(\thetastar - \widehat{\theta})}{\sqrt{u^T\Sigmastar u}}\right)\right|,
\end{align*}
where $g(t) := \int \sigma(s+t)h^*(s) ds$. Note that $g(t)$ is a monotonically increasing function, and by our restriction~\eqref{EqnRestrict}, we have $g'(0) =  \int \sigma'(s)h(s) ds = 1$. Thus, following the same argument used to establish inequality~\eqref{EqnLocParamUpper} above, we obtain 
\begin{equation*}
\|(\Sigma^*)^{-1}(\thetastar - \widehat{\theta})\|_2 \le C \left(2^L \sqrt{\frac{p}{n}} \vee \epsilon\right),
\end{equation*}
where the constant $C$ only depends on $B$. Since $\| \Sigma^* \|_2 \leq M_2$, we obtain 
$$\| \thetastar - \widehat{\theta}\|_2 \leq C' \left(2^L \sqrt{\frac{p}{n}} \vee \epsilon\right),$$
where the constant $C'$ depends on $B$ and $M_2$.

%%%%%

\section{Proofs for covariance estimation}

In this appendix, we derive the upper and lower bounds specific to covariance matrix estimation.

\subsection{Proof of Theorem~\ref{CovUpper}}
\label{AppCovUpper}

From Theorem~\ref{Upper1}, we immediately see that with at least $1-\delta$ probability, 
\begin{equation}
\sup_{f \in \mathcal{D}} \left|\E_{P_{\Sigmastar}}f(X) - \E_{P_{\widehat{\Sigma}}}f(X)\right| \leq 2 C\left(2^{L}\sqrt{\frac{p}{n}} + \sqrt{\frac{2\log(1/\delta)}{n}}\right)+ \frac{B}{2}\epsilon. \label{cov0}
\end{equation}

For the first layer, let $u \in \real^p$ be such that $\|u\|_2 = 1$, and define
\begin{equation*}
w^{(1)}_i = 
\begin{cases}
\frac{u}{\sqrt{u^T\Sigmastar u}}, & \text{if } i = 1, \\
0, & \text{otherwise},
\end{cases}
\end{equation*}
and
\begin{equation*}
b^{(1)}_i = 
\begin{cases}
1, & \text{if } i = 1, \\
0, & \text{otherwise}.
\end{cases}
\end{equation*}
(Note that since $\lambda_{\min}(\Sigma^*) \ge \frac{1}{B}$, we are guaranteed that $\|w^{(1)}\|_2 \le B$.)
For $l > 1$, define
\begin{equation*}
w^{(l)}_i =
\begin{cases}
(1,0, \dots,0), & \text{if } i = 1, \\
0, & \text{otherwise}.
\end{cases}
\end{equation*}
Under this setup, we have $f(X) = \sigma \left(\frac{u^TX}{\sqrt{u^T\Sigmastar u}} + 1\right).$

Note that for $X \sim \mathcal{N}(0, \Sigmastar)$, we have $\frac{u^TX}{\sqrt{u^T\Sigmastar u}} \sim \mathcal{N}(0,1)$. For $X \sim \mathcal{N}(0, \widehat{\Sigma})$, we have $\frac{u^TX}{\sqrt{u^T\Sigmastar u}} = \frac{\sqrt{u^T\widehat{\Sigma}u}}{\sqrt{u^T\Sigmastar u}} \frac{u^TX}{\sqrt{u^T\widehat{\Sigma} u}}$. Denoting $\Delta = \sqrt{\frac{u^T\widehat{\Sigma}u}{u^T\Sigmastar u}}$, we have $\frac{u^TX}{\sqrt{u^T\Sigmastar u}} = \Delta \frac{u^TX}{\sqrt{u^T\widehat{\Sigma} u}}$.

Let $Z$ be a random variable from $\mathcal{N}(0,1)$. From inequality~\eqref{cov0}, we have 
\begin{equation}
\sup_{u:\|u\|_2 = B} \left|\mathbb{E}\sigma(Z+1) - \mathbb{E}\sigma(\Delta Z + 1)\right| \leq 2 C\left(2^{L}\sqrt{\frac{p}{n}} + \sqrt{\frac{2\log(1/\delta)}{n}}\right)+ \frac{B}{2}\epsilon. \label{cov2}
\end{equation}

Define $g(t) =  \int_{-\infty}^{\infty}\sigma(tz+1)\phi(z)dz$, where $\phi(z)$ is density of $\mathcal{N}(0,1)$. Then
\begin{equation*}
g'(t) = \int_{-\infty}^{\infty} \sigma(tz+1)(1-\sigma(tz+1))z\phi(z)dz.
\end{equation*}
When $t > 0$, we can see that $g'(t) < 0$. Thus, $g(t)$ is a monotonically decreasing function on $(0, +\infty )$ and $g'(1) \neq 0$. Hence, when $|g(t) - g(1)|$ is sufficiently small, we obtain $|g(t) - g(1)| \geq c'|t-1|$ for some constant $c'$.

Thus, we have 
\begin{equation*}
\sup_{u:\|u\|_2=B} |\Delta - 1| \leq \frac{2C}{c'}\left(2^{L}\sqrt{\frac{p}{n}} + \sqrt{\frac{2\log(1/\delta)}{n}}\right)+ \frac{B}{2c'}\epsilon.
\end{equation*}
%Take $u$ as the eigenvector corresponding to largest eigenvalue of $\widehat{\Sigma}$, we get $\widehat{\lambda_1} \leq \lambda_1 (1+ 2 \frac{1}{c'}C(\sqrt{\frac{q\log q}{n}} + \sqrt{\frac{2\log(1/\delta)}{n}})+ 2 \frac{1}{c'}\epsilon)$. Similarly, we get $\widehat{\lambda_1} \geq \lambda_1 (1- 2 \frac{1}{c'}C(\sqrt{\frac{q\log q}{n}} + \sqrt{\frac{2\log(1/\delta)}{n}})+ 2 \frac{1}{c'}\epsilon)$.
%
%Thus we get 
%
%$$1- 2 \frac{1}{c'}C(\sqrt{\frac{q\log q}{n}} + \sqrt{\frac{2\log(1/\delta)}{n}})+ 2 \frac{1}{c'}\epsilon \leq \frac{\|\widehat{\Sigma}\|_2}{\|\Sigma\|_2} \lesssim 1+ 2 \frac{1}{c'}C(\sqrt{\frac{q\log q}{n}} + \sqrt{\frac{2\log(1/\delta)}{n}})+ 2 \frac{1}{c'}\epsilon $$
%Again, define $h(t)=\sqrt{t}$, when $|h(t) - h(1)|$ is sufficiently small, we have $|h(t) - h(1)| \geq c''|t-1|$. Thus let $t = \Delta^2$, we have 
Also note that for $|\Delta - 1| \le 1$, we have $|\Delta^2 - 1| \le 2|\Delta - 1|$. Furthermore, since
%%
%$$ |\Delta^2 - 1| \leq \frac{1}{c''}|\Delta - 1|$$
%Since 
\begin{equation}
\sup_{u:\|u\|_2=B} |\Delta^2 - 1| = \sup_{u:\|u\|_2=B} \left|\frac{u^T(\widehat{\Sigma}-\Sigmastar)u}{u^T\Sigmastar u}\right| = \left\|\Sigma^{*-\frac{1}{2}}(\widehat{\Sigma}-\Sigmastar)\Sigma^{*-\frac{1}{2}}\right\|_2, \label{cov1}
\end{equation}
with $\|\Sigmastar\|_2 \leq M$, we obtain $\left\| \Sigma^{*-\frac{1}{2}} (\widehat{\Sigma}-\Sigmastar)\Sigma^{*-\frac{1}{2}} \right\|_2 \geq M^{-1}\|\widehat{\Sigma}-\Sigmastar \|_2$. 
Hence, we can conclude that
$$\|\widehat{\Sigma}-\Sigmastar\|_2 \leq \frac{2MC}{c'c''}\left(2^{L}\sqrt{\frac{p}{n}} + \sqrt{\frac{2\log(1/\delta)}{n}}\right)+ \frac{MB}{2c'c''}\epsilon,$$
with probability at least $1-\delta$, implying that
$$\|\widehat{\Sigma}-\Sigma\|_2 \leq C' \left( 2^L \sqrt{\frac{p}{n}} \vee \epsilon \right) ,$$
with probability at least $1-e^{(p+n\epsilon^2)}$, where $C'$ only depends on $B$.

%%%%%

\subsection{Proof of Theorem~\ref{CovLower}}
\label{AppCovLower}

We will again use Theorem~\ref{ThmModulus} to obtain a lower bound, so our main effort is to calculate $\mathcal{M}(0)$ and $m(\epsilon, \Theta)$.

For the loss function $L(P_{\Sigma_1}, P_{\Sigma_2}) = \|\Sigma_1 - \Sigma_2\|_{2}$, we have $ m(\epsilon, \Theta) \geq \epsilon$.
Let $U=I_p$ and define $V$ such that $V_{11} = (1+\epsilon)^2$ and $V_{ii}=U_{ii}$ for $2 \le i \le p$. For Gaussian distributions, we may calculate
\begin{align*}
W_2(P_{U}, P_{V})  & = \left(\tr(U) + \tr(V) - 2\tr\left(\left(U^{1/2} V U^{1/2}\right)^{1/2}\right)\right)^{1/2} =  \|U^{1/2} - V^{1/2}\|_{F} = \epsilon.
\end{align*}
Since $W_1(P_{U}, P_{V}) \leq W_2(P_{U}, P_{V})$, this implies that $W_1(P_{U}, P_{V})  \leq \epsilon$. 
Thus,
\begin{align*} 
m(\epsilon, \Theta) & =  \sup_{W(P_{\Sigma_1}, P_{\Sigma_2})\leq \epsilon, \Sigma_1, \Sigma_2 \in \Theta} \|\Sigma_1 - \Sigma_2\|^2_{2} \geq \|U - V\|_{2} > \epsilon.
\end{align*}

The construction used in the proof of Theorem 6 of \cite{Ma1} can be used to obtain a bound on $\mathcal{M}(0)$. Specifically, using a variant of Fano's method, the authors implicitly provide a lower bound on the estimation error of the covariance matrix of the form 
$$ \inf_{\widehat{\Sigma}} \sup_{\Sigmastar \in \Theta'} \mathbb{P} \left( \|\Sigma^* - \widehat{\Sigma}\|_2^2 \geq C\lambda^2\frac{p}{n}\right) \ge c,$$
for any $\lambda > 0$, where the parameter space $\Theta'$ consists of positive semidefinite matrices $\Sigma$ satisfying $\lambda_{\min}(\Sigma) \geq \frac{1}{4}\lambda$ and $\lambda_{\max}(\Sigma) \leq \frac{3}{4}\lambda$. In fact, a similar argument shows that for any $\alpha \in \left(0, \frac{1}{4}\right]$, we have
\begin{equation*}
\inf_{\widehat{\Sigma}} \sup_{\Sigmastar \in \Theta'_\alpha} \mathbb{P} \left( \|\Sigma^* - \widehat{\Sigma}\|_2^2 \geq C_\alpha\lambda^2\frac{p}{n}\right) \ge c,
\end{equation*}
where
\begin{equation*}
\Theta'_\alpha := \left\{\Sigma \succeq 0: \lambda_{\min}(\Sigma) \ge \left(\frac{1}{2} - \alpha\right) \lambda, \quad \lambda_{\max}(\Sigma) \le \left(\frac{1}{2} + \alpha\right) \lambda\right\},
\end{equation*}
and the constant $C_\alpha$ depends only on $\alpha$.

For an appropriate choice of $\lambda$ and $\alpha$, e.g., $\alpha = \min\left\{\frac{M_2 - M_1}{2(M_1 + M_2)}, \frac{1}{4}\right\}$ and $\lambda = \frac{M_1}{1/2 - \alpha}$, we can make $\Theta'_\alpha$ a subset of our parameter space $\Theta$, implying a lower bound of $\mathcal{M}(0) = \Omega(\frac{p}{n})$ (where the constant prefactor depends on $M_1$ and $M_2$).

Thus, the final lower bound is $\Omega\left(\sqrt{\frac{p}{n}} \vee \epsilon\right)$.

%%%%%

\subsection{Proof of Theorem~\ref{SparseCovUpper1}}
\label{AppSparseCovUpper1}

For $\Sigmastar, \widehat{\Sigma} \in \mathcal{F}_1(k)$, we know by Lemma 2 of \cite{covlower1} that
$$ \|\widehat{\Sigma} - \Sigmastar\|_2 \leq \frac{3}{B^2} \max_{u \in \mathcal{U}_1(B, 2k)}|u^T(\widehat{\Sigma} - \Sigmastar)u|. $$
Thus, to obtain an upper bound, we can choose the same parameter setup as in the proof of Theorem~\ref{CovUpper}, since the weight parameters are $2k$-sparse. We define $g(t)$ in the same way as before. Inequality \eqref{cov2} then becomes
\begin{multline*}
\sup_{u: \|u\|_2 = B} \left|\mathbb{E}\sigma(z+1) - \mathbb{E}\sigma(\Delta z + 1)\right| \\
\leq 2 C\left(2^{L}\sqrt{\frac{\max(2k, \log p + 2k\log(\log{p}/(2k))}{n}} + \sqrt{\frac{2\log(1/\delta)}{n}}\right) + \frac{B}{2}\epsilon,
\end{multline*}
since by Lemma~\ref{VCdim_lemma}, the VC dimension of the function class \eqref{func6} is $O(\max(2k, \log p + 2k\log(\log{p}/(2k)))$. The remainder of the proof proceeds as before, implying the desired result.

%%%%%

\subsection{Proof of Theorem~\ref{SparseCovUpper2}}
\label{AppSparseCovUpper2}

For $\Sigmastar, \widehat{\Sigma} \in \mathcal{F}_2(k)$, we have
$$ \|\widehat{\Sigma} - \Sigmastar\|_2 = \frac{1}{B^2} \max_{u: \|u\|_2 = B, \|u\|_0 \le 2k}|u^T(\widehat{\Sigma} - \Sigmastar)u|. $$

Following the same argument used in the proof of Theorem~\ref{CovUpper}, we choose same parameter setup and define $g(t)$ as before. Inequality \eqref{cov2} then becomes
\begin{equation*}
\sup_{u: \|u\|_2 = B} \left|\E\sigma(z+1) - \E\sigma(\Delta z + 1)\right| \leq 2 C\left(2^{L}\sqrt{\frac{2k+2k\log\frac{p}{2k}}{n}} + \sqrt{\frac{2\log(1/\delta)}{n}}\right)+ \frac{B}{2}\epsilon,
\end{equation*}
by Theorem~\ref{Upper2}. The remainder of the argument holds as before, implying the desired result.

%%%%%

\subsection{Proof of Theorem~\ref{SparseCovLower}}
\label{AppSparseCovLower}

As in the proof for Theorem~\ref{CovLower}, we consider two diagonal covariance matrices, $U$ and $V$, where $U = I_p$ and $V$ is defined by $V_{11} = (1+\epsilon)^2$ and $V_{ii}=U_{ii}$ for $2 \le i \le p$. Note that both $U$ and $V$ are in $\mathcal{F}_1(k)$ and $\mathcal{F}_2(k)$. Thus, the same argument used in the proof of Theorem~\ref{CovLower} implies that $m(\epsilon, \Theta) \geq  \epsilon $.

For banded covariance matrices, we can use a construction from the proof of Theorem 3 in \cite{covlower1} to obtain a lower bound on $\mathcal{M}(0)$. For given positive integers $k$ and $m$ with $2k < p$ and $1 \leq m \leq k$, define the $p \times p$ matrix $B(m,k) = (b_{ij})_{p\times p}$ with 
$$b_{ij} = I\{i=m \,\text{and}\, m+1 \leq j \leq 2k, \text{or}\, j = m \,\text{and}\, m+1\leq i \leq 2k\}.$$
For a value of $\gamma > 0$ to be specified later, we define the following collection of $2^k$ covariance matrices: 
$$ \mathcal{F}_{11} = \left\{\Sigma(\theta): \Sigma(\theta) = I_p + \gamma \sum_{m=1}^k \theta_m B(m,k), \theta=(\theta_m) \in \{0,1\}^k\right\}.$$
Assouad's lemma~\citep{Yu97} can be used to obtain a minimax lower bound over this parameter space. Let $X_i \stackrel{i.i.d.}{\sim} N(0, \Sigma(\theta))$ with $\Sigma(\theta)\in \mathcal{F}_{11}$. %Denote the joint distribution by $P_\theta$.
By Assouad's lemma, we obtain a lower bound of the form
\begin{equation}
\label{EqnAssouad}
%\inf_{\widehat{\Sigma}} \max_{\theta \in \{0,1\}^k} 2^2 \E_\theta \|\widehat{\Sigma} - \Sigma(\theta)\|_2^2 \geq
\min_{H(\theta,\theta')\geq 1} \frac{\|\Sigma(\theta) - \Sigma(\theta')\|_2^2}{H(\theta, \theta')} \cdot \frac{k}{2}  \cdot \min_{H(\theta, \theta')=1} \|P_\theta \wedge P_{\theta'}\|,
\end{equation}
where $H$ is the Hamming distance and $\|P_\theta \wedge P_{\theta'}\| = 1 - \frac{\|P_\theta - P_{\theta'}\|_1}{2}$ is the TV affinity. We now bound the individual terms separately. For the first factor, consider a fixed pair $(\theta, \theta')$ and let $v \in \{0,1\}^p$ be defined such that $v_i = 1$ for $\frac{k}{2} \le i \le k$ and $v_i = 0$ otherwise. Also define $w=(\Sigma(\theta) - \Sigma(\theta'))v$. Note that $w$ has exactly $H(\theta, \theta')$ components of magnitude $\frac{k}{2}\gamma$, and $\|v\|_2^2 = \frac{k}{2}$. Thus,
$$ \|\Sigma(\theta) - \Sigma(\theta')\|_2^2 \geq \frac{ \|(\Sigma(\theta) - \Sigma(\theta'))v \|_2^2}{\|v\|_2^2} \geq \frac{H(\theta, \theta') (\frac{k}{2}\gamma)^2}{k/2} = H(\theta, \theta')\frac{k}{2}\gamma^2.$$
For the last factor in expression~\eqref{EqnAssouad}, Lemma 6 of \cite{covlower1} shows that $\min_{H(\theta, \theta') = 1} \|P_\theta \wedge P_{\theta'}\| \ge c$ for a constant $c > 0$.
%when $H(\theta, \theta') = 1$, we will show that
%
%\begin{align*}
%\| P_{\theta'} - P_\theta \|_1^2 & \leq 2KL(P_{\theta'} | P_\theta) \\
%& = 2n\left(\frac{1}{2}\tr(\Sigma(\theta')\Sigma^{-1}(\theta)) - \frac{1}{2} \log\det(\Sigma(\theta')\Sigma^{-1}(\theta)) - \frac{p}{2}\right) \\
%& \leq c\cdot nk\gamma^2
%\end{align*}
%
%for some small constant $c$, where $KL(\cdot|\cdot)$ denotes the KL-divergence.
%The first inequality comes from Pinsker's inequality and the second equation is by definition. For the last inequality, write
%
%$$ \Sigma(\theta') = D_1 + \Sigma(\theta)$$ 
%
%Then 
%
%$$\frac{1}{2} \tr(\Sigma(\theta')\Sigma^{-1}(\theta)) - \frac{p}{2} = \frac{1}{2} \tr(D_1\Sigma^{-1}(\theta)).$$
%
%By Taylor's expansion, 
%
%$$ \frac{1}{2} \log\det(\Sigma(\theta')\Sigma^{-1}(\theta)) =  \frac{1}{2} \log\det(I_p + D_1\Sigma^{-1}(\theta)) = \tr(D_1\Sigma^{-1}(\theta)) - R_3,$$
%
%where $R_3 \leq c_3\sum_{i=1}^{p} \lambda_i^2$ and the $\lambda_i$'s are the eigenvalues of $D_1\Sigma^{-1}(\theta)$. We write $\Sigma^{-1/2}(\theta) = UV^{1/2}U^T$, where $UU^T=I$ and $V$ is a diagonal matrix. It follows from the fact that the Frobenius norm of a matrix remains the same after an orthogonal transformation that 
%
%$$ \sum_{i=1}^{p}\lambda_i^2 = \|\Sigma^{-1/2}(\theta)D_1\Sigma^{-1/2}(\theta)\|_F^2 \leq \|V\|_2^2\cdot\|U^TD_1U\|_F^2 = \|\Sigma^{-1}(\theta)\|_2^2\cdot\|D_1\|_F^2 \leq c\cdot k \gamma^2.$$
%
%Thus, the third factor $ \|P_\theta \wedge P_{\theta'}\|_1 = 1 - \|P_{\theta'} - P_\theta\|_1^2$ can be lower bounded.
Finally, plugging $\gamma^2 = \frac{1}{kn}$ into our lower bounds for the expression~\eqref{EqnAssouad}, we obtain an overall bound of the form $\Omega\left(\frac{k}{n}\right)$.
%$\min_{H(\theta,\theta')\geq 1} \frac{\|\Sigma(\theta) - \Sigma(\theta')\|_2^2}{H(\theta, \theta')} \geq \frac{1}{2n}$ and $\min_{H(\theta, \theta')=1} \|P_\theta \wedge P_{\theta'}\| \geq 1 - c$. Thus 
%
%$$\inf_{\widehat{\Sigma}} \max_{\theta \in \{0,1\}^k} \E_\theta \|\widehat{\Sigma} - \Sigma(\theta)\|_2^2 \geq c'\frac{k}{n}.$$
%
Note that in the regime $n > k$, we also have
\begin{equation*}
\left\|\gamma \sum_{m=1}^k \theta_m B(m,k)\right\|_2 \leq \left\|\gamma \sum_{m=1}^k \theta_m B(m,k)\right\|_1 = O(1).
\end{equation*}
Thus, with an appropriate scaling (depending on $M_1$ and $M_2$), the parameter space $\mathcal{F}_{11}$ can be made a subset of our parameter space $\Theta_1(k)$. 

Next, define the set of matrices
$$\mathcal{F}_{12} = \left\{ \Sigma_m: \Sigma_m = I_p + \left(\sqrt{\frac{\log p}{n}} \cdot I\{i=j=m\}\right)_{p\times p}, \quad 0 \leq m \leq p\right\}.$$
Using Le Cam's method, we can obtain a minimax lower bound of the form $\Omega\left(\frac{\log p}{n}\right)$ over the class of covariance matrices in $\mathcal{F}_{12}$ (cf.\ Section 3.2.2 in \cite{covlower1}).
%
%$$\inf_{\widehat{\Sigma}} \sup_{\Sigma_m \in \mathcal{F}_{12}} \E \|\widehat{\Sigma} - \Sigma_m\|_2^2 \geq c\frac{\log p}{n}.$$
%
Similarly, with an appropriate scaling (depending on $M_1$ and $M_2$), the parameter space $\mathcal{F}_{12}$ can be made a subset of our parameter space $\Theta_1(k)$.

Combining the two bounds on $\mathcal{F}_{11}$ and $\mathcal{F}_{12}$, we then have
$$  \inf_{\widehat{\Sigma}}\sup_{\Sigma^* \in \Theta_1(k)} \mprob\left( \|\Sigma^* - \widehat{\Sigma}\|_2^2 \geq C\left(\frac{k}{n} + \frac{\log p}{n} \right)\right) \ge c.$$
Thus, we obtain $\mathcal{M}(0) = \Omega\left(\sqrt{\frac{k + \log p}{n}}\right)$, where the constant prefactor depends on $M_1$ and $M_2$.

For sparse matrix estimation, the proof of Theorem 4 in \cite{covlower2} exhibits a construction involving the set of rank-one matrices defined by
$$\Theta' = \left\{ \Sigma = I_p + \lambda vv^T: \|v\|_2 = 1, \|v\|_0=k \right\}, $$
showing that
$$  \inf_{\widehat{\Sigma}}\sup_{\Sigma^* \in \Theta'} \mprob\left(\|\Sigma^* - \widehat{\Sigma}\|_2^2 \geq C\left(\lambda \frac{k\log\frac{ep}{k}}{n} + \lambda^2\frac{k}{n}\right)\right) \ge c.$$
Note that we can change the $I_p$ in the definition of the parameter space $\Theta'$ to $\alpha I_p$, and a similar bound holds with constant prefactor also depending on $\alpha$. Clearly, we can choose $\lambda$ and $\alpha$ such that $\Theta_2$ is a superset of $\Theta'$. Thus, we have $\mathcal{M}(0) = \Omega\left(\sqrt{\frac{k + k\log \frac{ep}{k}}{n}}\right)$, where the constant prefactor depends on $M_1$ and $M_2$.

Thus, by Theorem~\ref{ThmModulus}, we obtain the desired lower bounds.

%%%%%

\subsection{Proof of Theorem~\ref{EllipCov}}
\label{AppEllipCov}

Note that the proof of Theorem~\ref{Upper1} only depended on the form of the sigmoid activation function insofar as its range is in $[0,1]$ and it has a bounded Lipschitz constant. These properties also hold for the ramp activation function. Thus, the same results hold for the function class $\mathcal{D}_e(B, L)$. With at least $1-\delta$ probability, we then have
\begin{equation*}
\sup_{f \in \mathcal{D}_e(B,L)} \left|\E_{P_{\Sigmastar}}f(X) - \E_{P_{\widehat{\Sigma}}}f(X)\right| \leq 2 C\left(2^{L}\sqrt{\frac{p}{n}} + \sqrt{\frac{2\log(1/\delta)}{n}}\right)+ \frac{B}{2}\epsilon.
\end{equation*}

For the first layer, let $u \in \real^p$ be such that $\|u\|_2 = 1$, and define
\begin{equation*}
w^{(1)}_i = 
\begin{cases}
\frac{u}{\sqrt{u^T\widehat{\Sigma} u}}, & \text{if } i = 1, \\
-\frac{u}{\sqrt{u^T\widehat{\Sigma} u}}, & \text{if } i = 2, \\
0, & \text{otherwise},
\end{cases}
\end{equation*}
and
\begin{equation*}
b^{(1)}_i = 
\begin{cases}
-\frac{1}{2}, & \text{if } i = 1, 2 \\
0, & \text{otherwise}.
\end{cases}
\end{equation*}
For the second layer, define
\begin{equation*}
w^{(2)}_i = 
\begin{cases}
\left(\frac{1}{2},-\frac{1}{2},0, \dots,0\right) & \text{if } i = 1, \\
0, & \text{otherwise}.
\end{cases}
\end{equation*}
For $l > 2$, define
\begin{equation*}
w^{(l)}_i =
\begin{cases}
(1,0, \dots,0), & \text{if } i = 1, \\
0, & \text{otherwise}.
\end{cases}
\end{equation*}
Under this setup, we have $f(X) = R\left(\left|\frac{u^TX}{\sqrt{u^T\widehat{\Sigma} u}}\right| \right).$ Note that when training the GAN model to estimate $\widehat{\Sigma}$, we can constrain the estimated parameters to lie in the true parameter space~\eqref{space:covGauss}. Thus, the weights in the first layer satisfy the constraint $\|w\|_1 \leq B$.

Thus, we have 
\begin{equation}
\left| \int R(\Delta|z|)h^*(z)dz  - \int R(|z|)\widehat{h}(z)dz \right| \leq 2 C\left(2^{L}\sqrt{\frac{p}{n}} + \sqrt{\frac{2\log(1/\delta)}{n}}\right)+ \frac{B}{2}\epsilon, \label{covellip}
\end{equation}
where $\Delta = \sqrt{\frac{u^T\Sigma^* u}{u^T\widehat{\Sigma} u}}$.

By the constraint, we also have
$$ \int R(|z|)h^*(z)dz =  \int R(|z|)\phi(z)dz
 =  \int R(|z|)\widehat{h}(z)dz,$$
implying that
\begin{equation*}
\left| \int R(\Delta |z|)h^*(z)dz  - \int R(|z|)h^*(t)dt \right| \leq 2 C\left(2^{L}\sqrt{\frac{p}{n}} + \sqrt{\frac{2\log(1/\delta)}{n}}\right)+ \frac{B}{2}\epsilon.
\end{equation*}

Defining $g(t) = \int R(t|z|)h^*(z)dz$, it can be shown that $g(t)$ is an increasing function for all $t>0$. Thus, the remainder of the proof follows as in Theorem~\ref{CovUpper}, and we obtain the desired result.

%%%%%

\section{Proofs for linear regression}

In this appendix, we derive the upper and lower bounds specific to linear regression.

\subsection{Proof of Theorem~\ref{RegUpper}}
\label{AppRegUpper}

By Theorem~\ref{Upper1}, we immediately see that with at least $1-\delta$ probability, 
$$\sup_{f \in \mathcal{D}} \left|\E_{P_{\betastar}}f(X, Y) - \E_{P_{\widehat{\beta}}}f(X, Y)\right| \leq 2 C\left(2^{L}\sqrt{\frac{p}{n}} + \sqrt{\frac{2\log(1/\delta)}{n}}\right)+ \frac{B}{2} \epsilon, $$ 
for any values of ${\betastar}$, $\Sigmastar$, and $P$.

For the first layer, let $u \in \real^p$ be such that $\|u\|_2 = B$, and define
\begin{equation*}
w^{(1)}_i = 
\begin{cases}
(-\beta^{*T},1), & \text{if } i = 1, \\
0, & \text{otherwise},
\end{cases}
\end{equation*}
and
\begin{equation*}
b^{(1)}_i = 
\begin{cases}
1, & \text{if } i = 1, \\
0, & \text{otherwise}.
\end{cases}
\end{equation*}
For $l > 1$, define
\begin{equation*}
w^{(l)}_i =
\begin{cases}
(1,0, \dots,0), & \text{if } i = 1, \\
0, & \text{otherwise}.
\end{cases}
\end{equation*}
Under this setup, we have $f(X,Y) = \sigma(Y - X^T\betastar + 1)$. Thus,
$$|\E_{P_{\betastar}}\sigma(Y-X^T\betastar+1) - \E_{P_{\widehat{\beta}}}\sigma(Y-X^T\betastar+1)| \leq 2 C\left(2^{L}\sqrt{\frac{p}{n}} + \sqrt{\frac{2\log(1/\delta)}{n}}\right)+ \frac{B}{2} \epsilon. $$

Let $Z = Y-X^T\betastar$ and define $g(t) = \E_Z \sigma(tZ + 1)$. Then we have
\begin{align*}
\left|g(1) - g\left(\sqrt{1 + (\widehat{\beta} - \betastar)^T \Sigmastar (\widehat{\beta} - \betastar)}\right)\right| & = \left| \mathbb{E}\sigma(Z + 1) - \mathbb{E}\sigma(Z + X^T(\widehat{\beta} - \betastar) + 1) \right| \\
& \leq 2 C\left(2^{L}\sqrt{\frac{p}{n}} + \sqrt{\frac{2\log(1/\delta)}{n}}\right)+ \frac{B}{2} \epsilon.
\end{align*}
Using the same arguments as in the proof of Theorem~\ref{CovUpper}, there exists $c_1>0$ such that $c_1|t-1| \leq |g(t) - g(1)|$. Similarly, $\sqrt{1+x} - 1 \ge c_2x$ for some constant $c_2$ when $\left|\sqrt{1+x} - 1\right|$ is sufficiently small. Thus, we have
$$ (\widehat{\beta} - \betastar)^T \Sigma (\widehat{\beta} - \betastar) \leq 2\frac{C}{c_1c_2}\left(2^{L}\sqrt{\frac{p}{n}} + \sqrt{\frac{2\log(1/\delta)}{n}}\right)+ \frac{B\epsilon}{2c_1c_2},$$
implying that
$$ \|\widehat{\beta} - \betastar\|_2^2 \leq 2\frac{C}{c_1c_2\lambda_{\min}(\Sigma)}\left(2^{L}\sqrt{\frac{p}{n}} + \sqrt{\frac{2\log(1/\delta)}{n}}\right)+ \frac{B\epsilon}{2c_1c_2\lambda_{\min}(\Sigma)},$$
which completes the proof.

%%%%%

\subsection{Proof of Theorem~\ref{RegLower}}
\label{AppRegLower}

We will apply Theorem~\ref{ThmModulus}. We will first derive the lower bound
%first compute a lower bound on $m(\epsilon, \Theta)$.
%Under the setup defined above, we claim that
$$m(\epsilon, \Theta) = \sup_{W(P_{\beta_1}, P_{\beta_2})\leq \epsilon} L(\beta_1, \beta_2) \geq \frac{\sqrt{\epsilon}}{\sigma}.$$
Let $\beta_1 = 0$ and $\beta_2 = \frac{\sqrt{\epsilon}}{\sigma} u$, where $u$ is a unit vector. 

Note that the joint distribution is $(X,Y)\sim \mathcal{N}(0, \Gamma)$, where 
\begin{equation*}
\Gamma = 
\begin{pmatrix}
      \sigma^2I & \sigma^2\beta \\
      \sigma^2\beta^T & \sigma^2\|\beta\|^2 + 1
\end{pmatrix}.      
\end{equation*}
Thus, the choices of $\beta_1$ and $\beta_2$ give
\begin{equation*}
\Gamma_1 = 
\begin{pmatrix}
      \sigma^2I & 0 \\
      0 & 1
\end{pmatrix}, \qquad \text{and }
\Gamma_2 = 
\begin{pmatrix}
      \sigma^2I & \sqrt{\epsilon} \sigma u \\
      \sqrt{\epsilon} \sigma u^T & \epsilon + 1
\end{pmatrix}.      
\end{equation*}
%Let's first prove $W_2(P_{\beta_1}, P_{\beta_2}) \leq \epsilon$:
We first compute
\begin{align*}
W_2(P_{\beta_1}, P_{\beta_2}) & = \tr(\Sigma_1) + \tr(\Sigma_2) - 2\tr((\Sigma_1^{1/2}\Sigma_2\Sigma_1^{1/2})^{1/2}) \\
& = \tr(\Sigma_1) + \tr(\Sigma_2) - 2\tr(M^{1/2}),
\end{align*}
where 
\begin{equation*}
M = 
\begin{pmatrix}
      \sigma^4I & \sqrt{\epsilon} \sigma^2 u \\
       \sqrt{\epsilon} \sigma^2 u^T & \epsilon+1
\end{pmatrix}.
\end{equation*}
The characteristic function of $M$ is $(\lambda-\sigma^4)^{p-1}(\lambda^2-(\sigma^4+\epsilon+1)\lambda + \sigma^4)$. Thus,
\begin{align*}
W_2(P_{\beta_1}, P_{\beta_2}) & = (p\sigma^2+1) + (p\sigma^2+1+\epsilon) - 2((p-1)\sigma^2+\sqrt{\sigma^4+2\sigma^2+1+\epsilon}) \\
& = 2(\sigma^2 + 1) +\epsilon - 2\sqrt{(\sigma^2+1)^2+\epsilon}\\
& \leq \epsilon.
\end{align*}
We also have $W_1(P_{\beta_1}, P_{\beta_2})  \leq W_2(P_{\beta_1}, P_{\beta_2}) \leq \epsilon$. Since
\begin{align*}
L(\beta_1, \beta_2) & = \|\beta_1 - \beta_2\|_2 = \left\|\frac{\sqrt{\epsilon}}{\sigma}u\right\|_2 = \frac{\sqrt{\epsilon}}{\sigma},
\end{align*}
we obtain the bound $m(\epsilon, \Theta) \geq \frac{\sqrt{\epsilon}}{\sigma}$.

For $\mathcal{M}(0)$, Example 15.14 of \cite{Wai19} shows that for any fixed $X$, we have
\begin{equation}
\label{EqnWaiReg}
\inf_{\widehat{\beta}} \sup_{\betastar \in \Theta} \mprob\left(\frac{1}{n}\|X(\widehat{\beta}-\betastar)\|_2^2 \geq \frac{1}{64}\frac{p}{n}\right) > \frac{1}{2}.
\end{equation}
%e.g., see . The probability is taken over $\widehat{\beta}$. Since we fix $X$, thus it is actually taken over the observed data $Y$. 
The construction used to derive the above result holds as long as the true parameter space $\Theta \subseteq \real^p$ contains the ball $\left\{\beta\in \real^p: \|\beta\|_2 \leq \frac{\sqrt{p}}{2}\right\}$. Thus, as long as $B_1 > \frac{\sqrt{p}}{2}$, we have the above result. 

We now use inequality~\eqref{EqnWaiReg} to derive a probabilistic lower bound on $\|\widehat{\beta} - \betastar\|_2^2$.
%Since $\frac{X^TX}{n} \rightarrow \sigma^2I$ in probability, we have $\forall c > 0, \exists N,$ such that
%$$ \mprob\left(\left\|\frac{X^TX}{n} - \sigma^2I\right\|_2 \leq \delta\right) > \frac{1}{2},$$
%and this holds for any $n > N$. Thus,
%%
%\begin{equation*}
%\lambda_{\max} \left(\frac{X^TX}{n}\right) \leq \sigma^2+ c. \label{reg11}
%\end{equation*}
Since $X$ has a Gaussian distribution, we have the following standard result~\citep{Wai19}: For any $\delta > 0$, 
\begin{equation*}
 \mprob\left(\left\|\frac{X^TX}{n} - \sigma^2I\right\|_2 \leq C\sqrt{\frac{p}{n}} \right) > 1 - \delta.
\end{equation*}
Thus, with high probability,
\begin{equation*}
\lambda_{\max} \left(\frac{X^TX}{n}\right) \leq \sigma^2+ C\sqrt{\frac{p}{n}}.
\end{equation*}

Taking $\delta = \frac{1}{2}$, we have
\begin{align*}
\mprob\left(\left(\sigma^2 + C\sqrt{\frac{p}{n}}\right)\|\widehat{\beta}-\betastar\|_2^2 \geq \frac{1}{64}\frac{p}{n}\right) & \geq \frac{1}{2} \mprob\left(\frac{1}{n}\|X(\widehat{\beta}-\betastar)\|_2^2 \geq \frac{1}{64}\frac{p}{n}\right) > \frac{1}{4}.
\end{align*}
Thus, 
$$ \inf_{\widehat{\beta}} \sup_{\betastar \in \Theta}\mprob\left(\|\widehat{\beta}-\betastar\|_2^2 \geq \frac{1}{64(\sigma^2 + C\sqrt{\frac{p}{n}})}\frac{p}{n}\right) > \frac{1}{4}.$$
As long as $n > pC^2$, we have 
\begin{align*}
\inf_{\widehat{\beta}} \sup_{\betastar \in \Theta}\mprob\left(\|\widehat{\beta}-\betastar\|_2^2 \geq \frac{1}{64(\sigma^2 + 1)}\frac{p}{n}\right) & >  \inf_{\widehat{\beta}} \sup_{\betastar \in \Theta}\mprob\left(\|\widehat{\beta}-\betastar\|_2^2 \geq \frac{1}{64(\sigma^2 + C\sqrt{\frac{p}{n}})}\frac{p}{n}\right) \\
& > \frac{1}{4}.
\end{align*}
%
%The first probability is taken over jointly observed data $X$ and $Y$, while the second probability is taken over just $Y$.
Thus, we can see that $\mathcal{M}(0) = \Omega\left(\sqrt{\frac{p}{n}}\right)$.

%%%%%

\subsection{Proof of Theorem~\ref{RegLower2}}
\label{AppRegLower2}

%%%%%%%%%%%          previous version %%%%%%%%
%Pick $\beta_1 = 0$, $\beta_2 = \sqrt{\epsilon} u$, where $u$ is a unit eigenvector of $\Sigma$, corresponding eigenvalue is $\lambda$. Then 
%\begin{equation*}
%\Gamma_1 = 
%\begin{pmatrix}
%      \Sigma & 0 \\
%      0 & 1
%\end{pmatrix},
%\Gamma_2 = 
%\begin{pmatrix}
%      \Sigma & \sqrt{\epsilon}\lambda u \\
%      \sqrt{\epsilon}\lambda u^T & \epsilon\lambda + 1
%\end{pmatrix}      
%\end{equation*}
%Let's first prove $W(P_{\beta_1}, P_{\beta_2}) \leq \epsilon$:
%\begin{align*}
%W(P_{\beta_1}, P_{\beta_2}) & = tr(\Gamma_1) + tr(\Gamma_2) - 2tr((\Gamma_1^{1/2}\Gamma_2\Gamma_1^{1/2})^{1/2}) \\
%& =  tr(\Gamma_1) + tr(\Gamma_2) - 2tr(B^{1/2})\\
%& = (tr(\Sigma)+1) + (tr(\Sigma)+1+\lambda \epsilon) - A \\
%%& = \epsilon + 4 - 2\sqrt{4+\epsilon}\\
%& \leq \epsilon
%\end{align*}
%where $B = \begin{pmatrix}
%      \Sigma^2 & \sqrt{\epsilon\lambda}\lambda u \\
%      \sqrt{\epsilon\lambda}\lambda u &  \epsilon\lambda + 1
%\end{pmatrix},$

We choose the same values for $\beta_1$ and $\beta_2$ as in the proof of Theorem~\ref{RegLower}. Then the loss is
\begin{align*}
L(\beta_1, \beta_2) & = \left(\E_X(X^T\beta_1 - X^T\beta_2)^2\right)^{1/2} = \left(\E_X(\sqrt{\epsilon}X^Tu)^2\right)^{1/2} = \sqrt{\epsilon}.
\end{align*}
Thus, we have $m(\epsilon, \Theta) \geq \sqrt{\epsilon}$. An expression for $\mathcal{M}(0)$ is provided by inequality~\eqref{EqnWaiReg} above.

\end{document}